\documentclass[a4paper]{article}
\usepackage{authblk}
\usepackage[english]{babel}
\usepackage[utf8]{inputenc}
\usepackage[T1]{fontenc}
\usepackage[a4paper,top=3cm,bottom=2cm,left=3cm,right=3cm,marginparwidth=1.75cm]{geometry}
\usepackage{titlesec}
\usepackage{bm}
\usepackage{amsmath}
\usepackage{amssymb}
\usepackage{mathrsfs}
\usepackage{graphicx}
\usepackage[colorinlistoftodos]{todonotes}
\usepackage[colorlinks=true, allcolors=blue]{hyperref}
\usepackage[numbers,sort&compress]{natbib}
\usepackage[amsmath,thmmarks]{ntheorem}
\usepackage{enumerate}
\usepackage{indentfirst}

\titleformat{\section}[hang]{\large\bfseries}{\thesection\quad}{0pt}{}[]
\titleformat{\subsection}[hang]{\normalsize\bfseries}{\thesubsection\quad}{0pt}{}[]

\theorembodyfont{\upshape}
\newtheorem{theorem}{Theorem}[section]
\newtheorem{lemma}{Lemma}[section]
\newtheorem{proposition}[theorem]{Proposition}

\newtheorem{definition}[theorem]{Definition}
\newtheorem{remark}[theorem]{Remark}
\newtheorem{assumption}[theorem]{Assumption}

\newenvironment{proof}{{\noindent\it Proof}\quad}{\hfill $\square$\par}
\numberwithin{equation}{section}

\begin{document}
\title{\textbf{Large and Moderate deviation for multivalued McKean-Vlasov stochastic differential equation}}
\author{Wei Liu, Fengwu Zhu}
\affil {School of Mathematics and Statistics, Wuhan University \\
Wuhan, Hubei 430072, PR China\\
wliu.math@whu.edu.cn,\ \ \ 2018202010044@whu.edu.cn}


\date{}

\maketitle

\noindent{\bf Abstract}
\quad In this paper, we present sufficient conditions and criteria to establish the large and moderate deviation principle for multivalued McKean-Vlasov stochastic differential equation by means of the weak convergence method.
\vskip0.3cm

\noindent {\bf Key words}{\quad Multivalued McKean-Vlasov SDE, Large deviation, Moderate deviation, Weak convergence method.}
\vskip0.3cm



\section{Introduction}

   Consider the following multivalued McKean-Vlasov stochastic differential equation (SDE in short):
\begin{equation}\label{mve}
\begin{cases}
	\mathrm{d}X_{t}\in b\left(X_{t},\mathcal{L}_{X_{t}}\right)\,\mathrm{d}t+\sigma \left(X_{t},\mathcal{L}_{X_{t}}\right)\,\mathrm{d}W_{t}-A(X_{t})\,\mathrm{d}t,\\
	X_{0}=x\in \overline{D(A)},
\end{cases}
\end{equation}
where $A:\mathbb{R} ^{d}\rightarrow 2^{\mathbb{R} ^{d}}$ is a multivalued maximal monotone operator with domain $D(A)$ which will be described below, and $ (W_{t})_{t\geq 0}$ is a d-dimensional standard Brownian motion on some filtered probability space $(\Omega ,\mathcal{F},P;\left( \mathcal{F}_{t}\right)_{t\geq 0})$. Here the coefficients $b$ and $\sigma$ depend not only on the current state $X_{t}$ but also on $\mathcal{L}_{X_t}$ which is the law of $X_t$. Notice that when the maximal monotone operator $A$ is the sub-differential of the indicator of a convex set, SDE (\ref{mve}) describes the reflected diffusion process.


The existence and uniqueness of the solution for equation (\ref{mve}) were studied by Chi \cite{CHM}. Under the global Lipschitz conditions on the coefficients $b$ and $\sigma$, Chi proved that there exists a unique pair of processes $(X_{t}, K_{t})$ so that
$$X_{t}=h+\int_{0}^{t}b\left( X_{s},\mathcal{L}_{X_{s}}\right)\,\mathrm{d}s+\int_{0}^{t}\sigma\left( X_{s},\mathcal{L}_{X_{s}}\right)\,\mathrm{d}W_{s}-K_{t},$$
where $K_{t}$ is a finite-variation process which will be explained in (\ref{kd}) below.

Note that when $A=0$, equation (\ref{mve}) is just the  classical McKean-Vlasov SDE, which was first suggested by Kac \cite{[Kac]} as a stochastic toy model for the Vlasov kinetic equation of plasma, and then introduced by McKean \cite{[MK]} to model plasma dynamics. These equations correspond to nonlinear Fokker-Planck equations in the field of partial differential equation and frequently appear in many nonlinear physical models such as Boltzmann equation and Landau equation. They also describe limiting behaviors of individual particles in an interacting particle system of mean-field type when the number of particles goes to infinity (so-called propagation of chaos). For this reason McKean-Vlasov SDEs are also referred as mean-field SDEs. The theory and applications of McKean-Vlasov SDEs and associated interacting particle systems have been extensively studied by a large number of researchers under various settings due to their wide range of applications in several fields, including physics, chemistry, biology, economics, financial mathematics etc., see \cite{[CGM],[GH],[GLWZ1],[GLWZ2],[LW],[LWZ],[MS],[TH]} and the references therein. For the existence and uniqueness of solution of McKean-Vlasov SDE, the reader is referred to \cite{[MK],[G],[MS],[SZ]} and recent works \cite{[MV],[HSS],[RZ]} as well as the references therein.

On the other hand, the equation (\ref{mve}) is called a multivalued SDE when the coefficients $b$ and $\sigma$ do not depend on $\mathcal{L}_{X_t}$. The existence and uniqueness of solution for multivalued SDE were first proved by C\'epa \cite{cepa,cepa1}. Zhang \cite{ZX} extended C\'epa’s result to the infinite-dimensional case and relaxed the Lipschitz assumption on $b$ to the monotone case. Xu \cite{xu} showed explicit solutions for multivalued SDEs. Ren et al. \cite{zhwu} obtained the exponential ergodicity of non-Lipschitz multivalued SDEs. The regularity of invariant measures of MSDEs was studied by Ren and Wu \cite{ren}.

We now consider the following small perturbation of equation (\ref{mve}):
\begin{equation}\label{spmve}
	\begin{cases}
		\mathrm{d}X_{t}^{\epsilon}\in b\left( X_{t}^{\epsilon},\mathcal{L}_{X_{t}^{\epsilon}}\right)\mathrm{d}t+\sqrt{\epsilon}\,\sigma\left( X_{t}^{\epsilon},\mathcal{L}_{X_{t}^{\epsilon}}\right)\mathrm{d}W_{t}-A(X_{t}^{\epsilon})\,\mathrm{d}t,\\
		X^{\epsilon}_{0}=x\in \overline{D(A)},\ \ \ \epsilon\in (0,1].
	\end{cases}
\end{equation}

Denote by $(X^{\epsilon}_{t}, K^{\epsilon}_{t})$ the solution of equation (\ref{spmve}). The main purpose of this paper is to establish the large and moderate deviation principle for the solution $X^{\epsilon}_{t}$ of equation (\ref{spmve}), in the space $\mathbb{S}:=C([0,T],\overline{D(A)})$, i.e. the asymptotic estimates of the probabilities $P(X_{t}^{\epsilon}\in B),\ B\in\mathcal{B}(\mathbb{S})$. Large deviation can provide an exponential estimate for tail probability (or the probability of a rare event) in terms of some explicit rate function. In the case of stochastic processes, the heuristics underlying large deviation theory is to identify a deterministic path around which the diffusion is concentrated with overwhelming probability, so that the stochastic motion can be seen as a small random perturbation of this deterministic path.

For the classic McKean-Vlasov SDEs, i.e. $A=0$, the large deviation principle (LDP in short) was studied by Herrmann et al. \cite{[HIP]} and Dos Reis et al. \cite{[RST]}. The approach in \cite{[HIP]} and \cite{[RST]} is to  first replace the distribution $\mathcal{L}_{X_{t}^{\epsilon}}$ of $X_t^{\epsilon}$ in the coefficients with a Dirac measure $\delta_{X^0(t)}$, and then to use discretization, approximation and exponential equivalence arguments. Recently the first author et al.
\cite{liuw} established the large and moderate deviations for McKean-Vlasov SDEs with jumps by means of the weak convergence method.

For the multivalued SDEs, i.e. the coefficients do not depend on the law of $X_t$, C\'epa \cite{cepa} established the
LDP in the one-dimensional case by contraction principle. Ren et al. \cite{zhang} established the LDP for multivalued SDEs with monotone drifts (in particular containing a class of SDEs with reflection in a convex domain) by using the weak convergence approach. A slightly more general case was considered by Ren et al. \cite{rwz}, also by the weak convergence approach.

The Freidlin–Wentzell type large and moderate deviation principles for the multivalued McKean-Vlasov SDEs seem not to be studied in the literature. Compared with the usual McKean-Vlasov SDEs and multivalued SDEs, most of the difficulties come from the presence of $\mathcal{L}_{X_t}$ in the coefficients $b$ and $\sigma$, and the finite-variation process $K_t$ respectively. One only knows that $K_t$ is continuous and could not prove any further regularity such as H\"older continuity. Therefore, the classical method of time-discretized method is almost not applicable. The weak convergence approach, developed by Dupuis and Ellis \cite{[DE]} (see also \cite{BD2019,BD2000,BDM2008,BDM2011}), is proved to be a powerful tool to establish large deviation principles for various dynamical systems, see e.g. \cite{zhang,renzhang,zhangren,rwz,hu,liuw} and the references therein. In this paper, we will use the weak convergence approach to establish the large and moderate deviation principles for the multivalued McKean-Vlasov SDEs which we have given our result of the previous part in the Master thesis \cite{zfw}. 

The rest of this paper is organized as follows. In Section 2, we recall some well-known facts about multivalued McKean-Vlasov SDEs and a criterion for Laplace principle. In Section 3, we present the main results and proofs.


\section{Preliminaries}\label{pre}

In this section, we first recall some basic notions and definitions. 

\subsection{Maximal Monotone Operators}

Let $ 2^{\mathbb{R} ^{d}} $ denote all the subsets of $ \mathbb{R} ^{d} $. A map $ A:\mathbb{R} ^{d}\rightarrow 2^{\mathbb{R} ^{d}} $ is called a multivalued operator. Denote

\begin{align*}
	D\left( A\right) &:=\{ x\in \mathbb{R} ^{d}:A\left( x\right) \neq \emptyset \} ,\\	
	Im\left( A\right) &:=\cup_{{x\in D(A)}} A\left( x\right) ,\\
	Gr\left( A\right) &:=\{ \left( x,y\right) \in \mathbb{R} ^{d}\times \mathbb{R} ^{d}:x\in \mathbb{R} ^{d},y \in A\left( x\right) \}.
\end{align*}


\begin{definition}\label{mo}
	
	(1) A multivalued operator $A$ is called \textit{monotone} if
	$$
	\langle x-x',y-y'\rangle \geq 0,\quad\forall \left( x,y\right),\left( x',y'\right) \in Gr\left( A\right).
	$$

	(2) A multivalued operator $A$ is called \textit{maximal monotone} if
	$$
	\left( x,y\right) \in Gr\left( A\right) \Leftrightarrow \langle x-x',y-y'\rangle \geq 0,\quad\forall \left( x',y'\right) \in Gr\left( A\right).
	$$
	
\end{definition}

Let $A$ be a \textit{multivalued maximal monotone operator} on $\mathbb{R}^{d}$.

It is known that Int$(D(A))$ and $\overline{D(A)}$ are convex subsets of $\mathbb{R}^{d}$ and
\begin{center} 	Int$(D(A))$=Int$(\overline{D(A)}).$ \end{center}


Given $T>0$, let
\begin{equation}\label{kd}
	\mathcal{V}=\left\{K\in C([0,T],\mathbb{R}^{d}):\, K \text{ is of finite variation and }K_{0}=0 \right\}.
\end{equation}
For any $K\in \mathcal{V}$ and $t\in [0,T]$, denote the variation of $K$ on $[0,t]$ by $|K|_{0}^{t}$ and  $|K|_{TV}:=|K|_{0}^{T}$. We introduce the following set of functions pair:
\begin{align*}
	\mathcal{A}:=&\big\{(X,K):X\in C\left([0,T],\overline{D(A)}\right),K\in \mathcal{V}\;\text{and}
	\\
	\phantom{=\;\;}&\text{for all}\;(x,y)\in Gr(A),\langle X_{t}-x,dK_{t}-ydt \rangle\geq0 \big\}.
\end{align*}

It is known from \cite{CHM,cepa1} that for any pairs of $(X_{t},K_{t}), (\tilde{X}_{t},\tilde{K}_{t})\in \mathcal{A}$,
\begin{equation}\label{mono}
	\langle X_{t}-\tilde{X}_{t},\mathrm{d}K_{t}-\mathrm{d}\tilde{K}_{t}\rangle\geq 0.
\end{equation}

We recall the following result due to C\'epa \cite{cepa1}, which will be used in the proofs in Section 3.
\begin{proposition}\label{multi}
	Assume that Int$(D(A))\neq \emptyset$. For any $a\in$ Int$(D(A))$, there exist constants $\lambda_{1}>0$ and $\lambda_{2},\lambda_{3}\geq0$ such that for any $(X,K)\in\mathcal{A}$ and $0\leq s<t\leq T$,
	\begin{equation}
		\int_{s}^{t}\langle X_{r}-a,\mathrm{d}K_{r}\rangle_{\mathbb{R}^{d}}\geq \lambda_{1}|K|_{s}^{t}-\lambda_{2}\int_{s}^{t}|X_{r}-a|\mathrm{d}r-\lambda_{3}(t-s).
	\end{equation}
    where $|K|_{s}^{t}$ denotes the total variation of $K$ on $[s,t]$.
\end{proposition}


\subsection{Solutions to multivalued Mckean-Vlasov stochastic differential equations}

\begin{proposition}\label{solu}
	Assume that there exists a stochastic basis $S:=\left(\Omega,\mathcal{F},P;(\mathcal{F}_{t})_{t\geq0}\right)$ and a d-dimensional standard Brownian motion $(W_{t})_{t\geq0}$, then a  pair of continuous and ($\mathcal{F}_{t}$)-adapted processes $\left(X,K\right)$ defined on $S$ is called the weak solution to (\ref{mve}) if
	\begin{enumerate} [(i)]
		\item $X_{0}$ has the initial law $\mathcal{L}_{X_{0}}$, and $\left(X_{.}(\omega),K_{.}(\omega)\right)\in \mathcal{A}$ for $P$-almost all $\omega\in\Omega$.
		\item It holds that $P$-almost surely
		$$\int_{0}^{t}\left(|b(X_{s},\mathcal{L}_{X_{s}})|+\lVert\sigma(X_{s},\mathcal{L}_{X_{s}})\rVert^{2}\right)\mathrm{d}s <+\infty$$
		\item
		$$X_{t}=X_{0}-K_{t}+\int_{0}^{t}b(X_{s},\mathcal{L}_{X_{s}})\mathrm{d}s+\int_{0}^{t}\sigma(X_{s},\mathcal{L}_{X_{s}})\mathrm{d}W_{s}, \quad \forall t\in[0,T],$$
		where $\mathcal{L}_{X_{s}}=P\circ X_{s}^{-1}.$
	\end{enumerate}
\end{proposition}

\begin{definition}\label{lawuni}
(Law Uniqueness) Let $ (S,W,(X,K)) $ and $ (\widetilde{S};\widetilde{W},(\widetilde{X},\widetilde{K})) $ be two weak solutions with the same initial law $ \mathcal{L}_{X_{0}}=\mathcal{L}_{\widetilde{X}_{0}} $. The law uniqueness holds for (\ref{mve}) if $ (X,K) $ and $ (\widetilde{X},\widetilde{K}) $ have the same law.
\end{definition}

\begin{definition}\label{pathuni}
	(Pathwise Uniqueness) Let $ (S,W,(X,K)) $ and $ (S,W,(\widetilde{X},\widetilde{K})) $ be two strong solutions with the same initial law $ X_{0}=\widetilde{X}_{0} $. The pathwise uniqueness holds for (\ref{mve}) if for all $ t\in[0,1] $, $ (X_{t},K_{t})=(\widetilde{X}_{t},\widetilde{K}_{t}) $.
\end{definition}

Let $\mathcal{P}_2(\mathbb{R}^d)$ be the set of all probability measures on $(\mathbb{R}^d,\mathcal{B}(\mathbb{R}^d))$, with finite second moments. $\mathcal{P}_2(\mathbb{R}^d)$ is a Polish space equipped with $L^2$ Wasserstein distance defined by
$$\mathbb{W}_{2}(\mu,\nu):=\inf\limits_{\xi\in \Pi(\mu_{1},\mu_{2})}\left(\iint_{\mathbb{R}^d\times\mathbb{R}^d}|x-y|^{2} \xi(\mathrm{d}x,\mathrm{d}y)\right)^{\frac{1}{2}},\ \ \forall \mu,\ \nu\in \mathcal{P}_2(\mathbb{R}^d)$$
where $\Pi(\mu,\nu)$ is the set of all couplings of $\mu,\nu$, i.e. the set of all probability measures on $\mathbb{R}^d\times\mathbb{R}^d$ with marginal distributions $\mu$ and $\nu$. It is well known that $\mathbb{W}_{2}$ is a complete metric on $\mathcal{P}_2(\mathbb{R}^d)$.

\begin{remark}\label{W2}
	From the definition above we know that for any $\mathbb{R}^{d}$-valued random variables $X$ and $Y$,
	$$\mathbb{W}_{2}\left(\mathcal{L}_{X},\mathcal{L}_{Y}\right)\leq \left[\mathbb{E}(X-Y)^{2}\right]^{\frac{1}{2}},$$ which will be often used in the proofs later.
\end{remark}

Next we represent the result about the existence and uniqueness of strong solution in \cite{CHM}.

\begin{proposition}\label{strongsolu}
	(\textbf{Strong solution}) Assume that Int$(D(A))\neq \emptyset$, $b:\mathbb{R}^{d}\times\mathcal{P}(\mathbb{R}^d)\rightarrow\mathbb{R}^{d}$ and $\sigma:\mathbb{R}^{d}\times\mathcal{P}(\mathbb{R}^d)\rightarrow\mathbb{R}^{d}\times\mathbb{R}^{d}$ are global Lipschitz continuous, i.e. there exist some positive constants $L_{b}$ and $L_{\sigma}$ such that for any $x,x'\in\mathbb{R}^d$ and $\mu,\nu\in \mathcal{P}(\mathbb{R}^d)$,
	$$|b(x,\mu)-b(x',\nu)|\leq L_{b}(|x-x'|+\mathbb{W}_{2}(\mu,\nu)),$$
	and
	$$\lVert\sigma(x,\mu)-\sigma(x',\nu)\rVert\leq L_{\sigma}(|x-x'|+\mathbb{W}_{2}(\mu,\nu)).$$
	Then for any $\mathcal{F}_{0}$-measurable random variable $X_{0}$ with $\mathbb{E}|X_{0}|^{2}<+\infty$, there exists a unique $\mathcal{F}_{t}$-adapted solution $(X,K)\in \mathcal{A}$ to equation (\ref{mve}) so that
	$$X_{t}=X_{0}-K_{t}+\int_{0}^{t}b(X_{s},\mathcal{L}_{X_{s}})\mathrm{d}s+\int_{0}^{t}\sigma(X_{s},\mathcal{L}_{X_{s}})\mathrm{d}W_{s}, \quad \mathcal{L}_{X_{s}}=P\circ X_{s}^{-1}.$$
\end{proposition}

\subsection{Large Deviation Principle}\label{ldp}

 With the fundamental result about the existence and uniqueness of strong solution to multivalued stochasitic Mckean-Vlasov equation, now we are able to state precisely the goal of this paper. First we recall an abstract criterion for Laplace principle, which is equivalent to the LDP with the same rate function through the weak convergence method (see \cite{BD2000,BDM2008,zhang}).

It is well known that there exists a Hilbert space $\mathbb{U}$ so that $l^{2}\subset \mathbb{U}$ is Hilbert-Schmidt with embedding operator J and $\left\{W_{t}^{k},k\in \mathbb{N}\right\}$ is a Brownian motion with values in $\mathbb{U}$ and its covariance operator is given by $Q=J\circ J^{*}$. For example one can take $U$ as the completion of $l^{2}$ with respect to the norm generated by scalar product
$$\left\langle f,f'\right\rangle_{\mathbb{U}}:=(\sum_{k=1}^{\infty}\frac{f_{k}f'_{k}}{k^{2}})^{\frac{1}{2}},\quad f,f'\in l^{2}.$$

Define
\begin{equation}
	\ell_{T}^{2}:=\left\{\int_{0}^{.}f(s)\mathrm{d}s:f\in L^{2}([0,T],l^{2})\right\}.
\end{equation}
with the norm $$\lVert f\rVert_{\ell_{T}^{2}}:=\left(\int_{0}^{T}\lVert f(s)\rVert_{l^{2}}^{2}\right)^{\frac{1}{2}},$$

For each $m,T\geq0$, define
\begin{equation}
	\mathcal{S}_{m}:=\left\{f\in L^{2}([0,T],l^{2}):\lVert f\rVert_{\ell_{T}^{2}}^{2}\leq m\right\}.
\end{equation}
$\mathcal{S}_{m}$ will be equipped with the topology of weak convergence in $L^{2}([0,T],l^{2})$, then $S_{m}$ is a  compact subset of $L^{2}([0,T],l^{2})$.

Set $\mathcal{S}=\cup_{{m\in(0,\infty)}}\mathcal{S}_{m}$ and
	\begin{align*}
		\mathcal{D}_{m}^{T}:=\bigg\{&F=\int_{0}^{.}f(s)\mathrm{d}s:[0,T]\rightarrow l^{2} \text{ is a continuous and }(\mathcal{F}_{t})\\
		&\text{-adapted process, and }f(\cdot,\omega)\in\mathcal{S}_{m}\text{ for almost all }\omega\bigg\}.
	\end{align*}
	
let $\mathbb{S}$ be a Polish space, we denote by $\mathcal{B}(\mathbb{S})$ the Borel $\sigma$-field, and by $C([0,T],\mathbb{U})$ the continuous function space from $[0,T]$ to $\mathbb{U}$, and endow it with the uniform distance so that $C([0,T],\mathbb{U})$ is also a Polish space. Let $\mu$ be the law of the Brownian motion $W$ in $C([0,T],\mathbb{U})$, then $(\mathcal{S}_{m},C([0,T],\mathbb{U}),\mu)$ forms an abstract Wiener space.

Let us recall the definition of a rate function and LDP.

\begin{definition} [Rate function]
	A function $I:\mathbb{S}\rightarrow[0,\infty]$ is called a rate function if for every $a<\infty$, the level set $\left\{g\in\mathbb{S}:I(g)\leq a\right\}$ is a closed subset of $\mathbb{S}$, and it is a good rate function if all the level sets defined above are compact subsets of $\mathbb{S}$.
\end{definition}

\begin{definition}[Large deviation principle] Let $I$ be a rate function on $\mathbb{S}$. Given a collection $\{\hbar(\epsilon)\}_{\epsilon>0}$ of
positive reals, a family $\left\{X^{\epsilon}\right\}_{\epsilon>0}$ of $\mathbb{S}$-valued random elements is said to satisfy a LDP on $\mathbb{S}$ with speed $\hbar(\epsilon)$ and rate function $I$ if the following two claims hold.
\begin{itemize}
  \item[(a)]
 (Upper bound) For each closed subset $C$ of $\mathbb{S},$
$$
\limsup_{\epsilon \rightarrow 0} \hbar(\epsilon) \log {P}\left(X^{\epsilon} \in C\right) \leq - \inf_{x \in C} I(x).
$$
 \item[(b)] (Lower bound) For each open subset $O$ of $\mathbb{S},$
$$
\liminf_{\epsilon \rightarrow 0} \hbar(\epsilon) \log {P}\left(X^{\epsilon} \in O\right) \geq - \inf_{x \in O} I(x).
$$
\end{itemize}
\end{definition}

Let $\left\{\mathcal{G}^{\epsilon}\right\}_{\epsilon>0}$ be a family of measurable maps from $C([0,T],\mathbb{U})$ to $\mathbb{S}$, and $X^\epsilon=\mathcal{G}^{\epsilon}(\sqrt{\epsilon}W)$. Next we introduce the following conditions which will be sufficient to establish the large and moderate deviation principles for $\left\{X^{\epsilon}\right\}_{\epsilon>0}$.

\begin{assumption}
	There exists a measurable map $\mathcal{G}^{0}:\ell_{T}^{2}\mapsto \mathbb{S}$ such that the following hold.
	\begin{enumerate}
		\item[(\textbf{LD}$_{\bm{1}}$)] For $m\in(0,\infty)$, let $\left\{u_{n},n\in\mathbb{N}\right\},u\in \mathcal{S}_{m}$ be such that $u_{n}\rightarrow u$ weakly as $n\rightarrow\infty$. Then
		$$\mathcal{G}^{0}\left(\int_{0}^{\cdot}u_{n}(s)\mathrm{d}s\right)\rightarrow\mathcal{G}^{0}\left(\int_{0}^{\cdot}u(s)\mathrm{d}s\right)\quad\quad \text{in }\mathbb{U}.$$
		\item[(\textbf{LD}$_{\bm{2}}$)] For $m\in(0,\infty)$, let a family $\left\{u_{\epsilon},\epsilon>0\right\}\subset\mathcal{D}_{m}^{T}$ converges in distribution to $h\in\mathcal{D}_{m}^{T}$ as $ \epsilon\rightarrow 0 $. Then
		$$\mathcal{G}^{\epsilon}\left(\sqrt{\epsilon}W+\int_{0}^{\cdot}u_{\epsilon}(s)\mathrm{d}s\right)\Rightarrow \mathcal{G}^{0}\left(\int_{0}^{\cdot}u(s)\mathrm{d}s\right)\quad \quad \text{in distribution}.$$
	\end{enumerate}

For each $g\in\mathbb{S}$, define $\mathbb{V}_{g}=\left\{u\in\mathcal{S}:g=\mathcal{G}^{0}\left(\int_{0}^{\cdot}u(s)\mathrm{d}s\right)\right\}$ and
\begin{equation}\label{rate}
	I(g):=\frac{1}{2}\inf\limits_{u\in\mathbb{V}_{g}}\lVert u\rVert_{\ell_{T}^{2}}^{2},
\end{equation}
By convention, $I(g)=\infty$ if $\mathbb{V}_{g}=\emptyset$. Then under (\textbf{LD}$_{\bm{2}}$), $ I(g) $ is a rate function.
\end{assumption}

We recall the following result due to \cite{BD2000} (see also \cite{BDM2008,BDM2011}).
\begin{theorem}
	 Under (\textbf{LD}$_{\bm{1}}$) and (\textbf{LD}$_{\bm{2}}$), $\left\{X^{\epsilon}\right\}_{\epsilon>0}$ satisfies the Laplace principle with  the rate function $I(g)$ given by (\ref{rate}). More precisely, for each real bounded continuous function $ f $ on $ \mathbb{S} $,
	\begin{equation}
		\lim\limits_{\epsilon\rightarrow 0}\epsilon\, \log\,\mathbb{E}^{\mu}\left(\exp[-f(X^{\epsilon})/\epsilon]\right)=-\inf\limits_{g\in\mathbb{V}_{g}}\left\{f(g)+I(g)\right\}.
	\end{equation}
	In particular, the family of $\left\{X^{\epsilon}\right\}_{\epsilon>0}$ satisfies the large deviation principle in $(\mathbb{S},\mathcal{B}(\mathbb{S}))$ with speed $\epsilon$ and rate function $I(g)$.
\end{theorem}


\section{Main Results and Proofs}\label{res}

Let $\epsilon>0$. For measurable maps$$b_{\epsilon}: \overline{D(A)} \times \mathcal{P}_{2} \rightarrow \mathbb{R}^{d},\quad \sigma_{\epsilon}: \overline{D(A)} \times \mathcal{P}_{2} \rightarrow \mathbb{R}^{d}\otimes \mathbb{R}^{d},$$
and $K_{t}^{\epsilon}\in \mathcal{V}$, 
consider the following small perturbation of multivalued Mckean-Vlasov SDEs:
\begin{equation}\label{per}
	\left\{
	\begin{array}{lr}
		X_{t}^{\epsilon} = h + \int_{0}^{t} b_{\epsilon}\left(X_{s}^{\epsilon},\mathcal{L}_{X_{s}^{\epsilon}} \right)\,\mathrm{d}s + \sqrt{\epsilon}\int_{0}^{t} \sigma_{\epsilon}\left(X_{s}^{\epsilon},\mathcal{L}_{X_{s}^{\epsilon}} \right)\,\mathrm{d}W_{s} - K_{t}^{\epsilon}, \\
		X_{0}^{\epsilon} = h \in\overline{D(A)},\quad \epsilon \in \left(0,1\right).
	\end{array}
	\right.
\end{equation}

\begin{assumption}
	We assume that
	
\textbf{(H0)}
	$ A $ is a mutivalued monotone operator with nonempty interior, i.e., Int$ D(A)\neq\emptyset $;
	
	\textbf{(H1)} $b$ and $\sigma$ are measurable functions such that for some $L>0$ and all $x,x'\in\mathbb{R}^d$ and $\mu,\nu\in \mathcal{P}(\mathbb{R}^d)$,
	\begin{align*}\label{h1}
		&\langle x-x',b(x,\mu)-b(x',\mu)\rangle\leq L\lvert x-x'\rvert^{2}, \\
		&\lvert b(x,\mu)-b(x,\mu')\rvert\leq L\mathbb{W}_{2}(\mu,\mu'), \\
		&\lvert b(x,\mu)-b(x',\mu)\rvert\leq L\left(1+|x|^{q}+|x'|^{q}\right)\lvert x-x'\rvert, \\
		&\lVert \sigma(x,\mu)-\sigma(x',\mu')\rVert_{\mathcal{L}_{2}}\leq L\left(\lvert x-x'\rvert+\mathbb{W}_{2}(\mu,\mu')\right), \\
		&\int_{0}^{T}\left(|b(0,\delta_{0})|+\lVert\sigma(0,\delta_{0})\rVert^{2}\right)\,\mathrm{d}t<\infty;
	\end{align*}
	
	\textbf{(H2)}
	As $\epsilon\rightarrow0$, the maps $b_{\epsilon}$ and $\sigma_{\epsilon}$ converge uniformly to $b$ and $\sigma$ respectively, in other words: there exist nonnegative constants $\rho_{b,\epsilon}$ and $\rho_{\sigma,\epsilon}$ converging to 0 as $\epsilon\rightarrow0$ such that
	\begin{align*}\label{h2}
		&\sup\limits_{(x,\mu)\in \mathbb{R}^{d}\times\mathcal{P}_{2}}\left(\lvert b_{\epsilon}(x,\mu)-b(x,\mu)\rvert\right)\leq \rho_{b,\epsilon}\,, \\
		&\sup\limits_{(x,\mu)\in \mathbb{R}^{d}\times\mathcal{P}_{2}}\left(\lVert \sigma_{\epsilon}(x,\mu)-\sigma(x,\mu)\rVert\right)\leq \rho_{\sigma,\epsilon}\,;
	\end{align*}
	
	\textbf{(H3)}\label{h3}
	For any $\epsilon> 0$, there exist a solution to (\ref{per}) as stated in Definition \ref{solu};
	
	\textbf{(H4)}
	Pathwise uniqueness holds for (\ref{per}) with fixed $\mathcal{L}_{X^{\epsilon}}$, i.e. if for any two solutions $(X_{1},\mathcal{L}_{X^{\epsilon}},K^{\epsilon})$ and $(X_{2},\mathcal{L}_{X^{\epsilon}},K^{\epsilon})$ of (\ref{per}),
	$$X_{1}(t)=X_{2}(t),\quad t\in[0,T],\,P\text{-}a.s.$$
	
\end{assumption}

\begin{proposition}
	Assume that (\hyperref[h1]{H1}) holds. There exists a unique function $X^{0}=\left\{(X^{0}(t),K^0(t)),t\in[0,T]\right\}$ such that
	\begin{enumerate}
		\item $X^{0}\in C\big([0,T],\overline{D(A)}\big)$,
		\item $\int_{0}^{T}\lvert b(X_{s}^{0},\mathcal{L}_{X_{s}^{0}})\rvert\mathrm{d}s<\infty,\,K_{t}^{0}\in \mathcal{V},\, \forall t\in[0,T]$.
		\item $X^{0}$ satisfies
		\begin{equation}\label{x0}
			X_{t}^{0}=h+\int_{0}^{t}b(X_{s}^{0},\mathcal{L}_{X_{s}^{0}})\mathrm{d}s-K_{t}^{0},\quad\forall t\in[0,T].
		\end{equation}
	\end{enumerate}
And $(X^{0}(t),K^0(t))$ is called the solution to (\ref{x0}).
\end{proposition}

\begin{remark} Since
	$\mathcal{L}_{X_{s}^{0}}=\delta_{{X}_{s}^{0}}$, we will denote $X^{0}$ by the unique solution to (\ref{x0}) for convenience.
\end{remark}

\begin{remark}\label{r0}
	(\hyperref[h1]{H1}) implies that for any $\epsilon>0$, there exists some $\rho>0$ such that for any $A\in \mathcal{B}\left([0,T]\right)$ with $Leb_{T}(A)<\rho$,
	\begin{equation}
		\int_{A}\lvert b(0,\delta_{0})\rvert\mathrm{~d}r +\int_{A}\lvert \sigma(0,\delta_{0})\rvert^{2}\mathrm{~d}r \leq \epsilon.
	\end{equation}
\end{remark}

To guarantee the Girsanov theorem, we need the following lemma which is stated in \cite[Lemma 2.21]{liuw},
\begin{lemma}\label{girs}
	Assume that $X_{t}^{\epsilon}$ is a solution of (\ref{per}) with initial value $X_{0}^{\epsilon}=h\in\overline{D(A)}$ and \textbf{(H4)} holds, then
	\begin{enumerate}
		\label{a0}\item[\textbf{(A0)}] for any $\mathcal{L}_{X^{\epsilon}}$ fixed, the maps $b_{\epsilon}(\cdot,\mathcal{L}_{X^{\epsilon}}):\mathbb{R}^{d}\rightarrow\mathbb{R}^{d}$ are $\mathcal{B}(\mathbb{R}^{d})/\mathcal{B}({\mathbb{R}^{d}})$-measurable and  $\sigma_{\epsilon}(\cdot,\mathcal{L}_{X^{\epsilon}}): \mathbb{R}^{d}\rightarrow\mathbb{R}^{d}\otimes\mathbb{R}^{d}$ are $\mathcal{B}(\mathbb{R}^{d})/\mathcal{B}({\mathbb{R}^{d}\otimes \mathbb{R}^{d}})$-measurable;
		\label{a1}\item[\textbf{(A1)}] (\ref{per}) has a unique solution $X^{\epsilon}$ as stated in Definition \ref{solu}.
	\end{enumerate}
Then there exists a map $\mathcal{G}_{\mathcal{L}_{X^{\epsilon}}}^{\epsilon}$ such that $X^{\epsilon}=\mathcal{G}_{\mathcal{L}_{X^{\epsilon}}}^{\epsilon}(\sqrt{\epsilon}W)$. For any $m\in(0,\infty)$, $u_{\epsilon}\in\mathcal{D}^{T}_m$, let
\begin{equation}\label{zep}	Z^{u_{\epsilon}}:=\mathcal{G}_{\mathcal{L}_{X^{\epsilon}}}^{\epsilon}\big(\sqrt{\epsilon}W+\int_{0}^{\cdot}u_{\epsilon}(s)\mathrm{d}s\big),
\end{equation}
and $Z^{u_{\epsilon}}$ is the unique stochastic process satisfying:
\begin{enumerate}
	\item[(1)] $Z^{u_{\epsilon}}$ is $\mathcal{F}_{t}$-adapted,
	\item[(2)]
	\begin{align*}
		\int_{0}^{T} \lVert b\left(Z_{s}^{u_{\epsilon}},\mathcal{L}_{X_{s}^{\epsilon}}\right) \rVert \,\mathrm{d}s + \int_{0}^{T} \lVert \sigma\left(Z_{s}^{u_{\epsilon}},\mathcal{L}_{X_{s}^{\epsilon}}\right) \rVert^{2} \,\mathrm{d}s + \int_{0}^{T} \lVert \sigma\left(Z_{s}^{u_{\epsilon}},\mathcal{L}_{X_{s}^{\epsilon}}\right)u_{\epsilon}\left(s\right) \rVert \,\mathrm{d}s +\left| K_{t}^{u_{\epsilon}} \right|_{0}^{T} <\infty,
	\end{align*}
	\item[(3)]
	\begin{equation}
	    \begin{split}
	    	Z_{t}^{u_{\epsilon}} = h + \int_{0}^{t} b_{\epsilon}\left(Z_{s}^{u_{\epsilon}},\mathcal{L}_{X_{s}^{\epsilon}} \right)\,\mathrm{d}s + \sqrt{\epsilon}\int_{0}^{t} \sigma_{\epsilon}\left(Z_{s}^{u_{\epsilon}},\mathcal{L}_{X_{s}^{\epsilon}} \right)\,\mathrm{d}W_{s} + \int_{0}^{t} \sigma_{\epsilon}\left(Z_{s}^{u_{\epsilon}},\mathcal{L}_{X_{s}^{\epsilon}} \right)u_{\epsilon}(s)\,\mathrm{d}s - K_{t}^{u_{\epsilon}}.
	    \end{split}
	\end{equation}
\end{enumerate}
\end{lemma}

\subsection{Large deviation principle}

\begin{proposition}\label{Yu}Assume that (\hyperref[h1]{\textbf{H1}}) holds. Then for any $ u\in \mathcal{S}_{m} $, there exists a unique function $ Y^{u}=\left\{ Y_{t}^{u}, t\in\left[ 0, T \right] \right\} $ satisfying
	\begin{enumerate}
		\item[(1)] $Y^{u}\in C\big( \left[ 0, T \right], \overline{D(A)} \big),$
		\item[(2)] $$\int_{0}^{T} \left| b\left(Y_{s}^{u},\mathcal{L}_{X_{s}^{0}}\right) \right| \,\mathrm{d}s + \int_{0}^{T} \left| \sigma\left(Y_{s}^{u},\mathcal{L}_{X_{s}^{0}}\right)u\left(s\right) \right| \,\mathrm{d}s + \left| K_{t}^{u} \right|_{0}^{T} <\infty, $$
		\item[(3)]
		\begin{equation}\label{yy}
			Y_{t}^{u}=h+\int_{0}^{T}b\left( Y_{s}^{u},\mathcal{L}_{X_{s}^{0}}\right) \,\mathrm{d}s + \int_{0}^{T}\sigma\left( Y_{s}^{u},\mathcal{L}_{X_{s}^{0}}\right)u\left(s\right) \,\mathrm{d}s-K_{t}^{u},\quad t\in \left[0,T\right].
		\end{equation}
	\end{enumerate}
	
	Moreover, for any $m>0$,
	$$ \sup\limits_{u\in \mathcal{S}_{m}}\sup\limits_{t\in \left[0,T\right]}\left| Y_{t}^{u}\right| <\infty.$$
	$Y^{u}$ is called the solution to the  equation (\ref{yy}) above.
	
\end{proposition}
\begin{proof} This result can be proved by using Proposition \ref{multi} presented in Section 2 and Proposition 3.7 in \cite{liuw}. We omit the tedious proofs here.\end{proof}

Next we present our main result about the large deviation principle.
\begin{theorem}
	Assume that (\hyperref[h1]{\textbf{H1}}), (\hyperref[h2]{\textbf{H2}}), \textbf{(H3)} and \textbf{(H4)} hold. Then the family of processes $ \left\{X_{t}^{\epsilon}, \epsilon >0, t\in[0,T]\right\} $ satisfies the large deviation principle on $C\big(\left[0,T\right],\overline{D(A)}\big)$ with speed $\epsilon$ and the good rate function $I$ given by
	\begin{equation}
		I\left( g\right): = \frac{1}{2}\inf\limits_{\left\{u\in \mathcal{S}_{m}: g=Y^{u}\right\}}\int_{0}^{T}\lVert u(s)\rVert^{2}\,\mathrm{d}s,\quad g\in C\left(\left[0,T\right],\overline{D(A)}\right),
	\end{equation}
where $Y^{u}$ is the solution to equation (\ref{yy}).
\end{theorem}

\begin{proof}
	By Proposition \ref{Yu} we can define a map $$ \mathcal{G}^{0}: \mathcal{S}\ni u\mapsto Y^{u} \in C\left(\left[0,T\right],\overline{D(A)}\right) $$ where $Y^{u}$ is the unique solution of eqution (\ref{yy}).

	Under the assumption (\hyperref[h1]{\textbf{H1}}), for any $\epsilon >0$, $m\in \left(0,\infty\right)$and $u_{\epsilon}\in \mathcal{D}_{m}^{T}$, by Lemma \ref{girs} there exists a unique solution $ \left\{Z_{t}^{u_{\epsilon}} ,t\in \left[0,T\right]\right\} $ to the following multivalued SDE
	\begin{equation}\label{LDP}
	\mathrm{d}Z_{t}^{u_{\epsilon}}\in b_{\epsilon}\left(Z_{t}^{u_{\epsilon}},\mathcal{L}_{X_{t}^{\epsilon}}\right)\mathrm{d}t+ \sqrt{\epsilon}\sigma_{\epsilon}\left(Z_{t}^{\epsilon},\mathcal{L}_{X_{t}^{\epsilon}}\right)\mathrm{d}W_{t} + \sigma_{\epsilon}\left(Z_{t}^{\epsilon},\mathcal{L}_{X_{t}^{\epsilon}}\right)u_{\epsilon}\left(t\right)\mathrm{d}t-A\left(X_{t}^{\epsilon}\right)\mathrm{d}t,
	\end{equation}
	with the initial data $Z_{0}^{u_{\epsilon}}=h$ and $X^{\epsilon}$ is the solution to (\ref{per}).
	
To establish the LDP by using Theorem 2.11, it is sufficient to verify the following two conditions:

	\textbf{(LDP)$\bm{_{1}}$} For any given $m\in\left(0,\infty\right)$, if $\left\{u_{n}, n\in \mathbb{N}\right\}\subset \mathcal{S}_{m}$ weakly converges to $u\in \mathcal{S}_{m}$ as $n\to \infty$. Then
	$$
	\lim\limits_{n\to \infty}\sup\limits_{t\in \left[0,T\right]}\lvert \mathcal{G}^{0}\left(u_{n}\right)\left(t\right)-\mathcal{G}^{0}\left(u\right)\left(t\right) \rvert=0.
	$$
	
	\textbf{(LDP)$\bm{_{2}}$} For any given $m\in\left(0,\infty\right)$, let $\left\{u_{\epsilon}, \epsilon>0\right\}\subset \mathcal{S}_{m}$. Then
	$$
	\lim\limits_{\epsilon\to 0}\mathbb{E}\left(\sup\limits_{t\in \left[0,T\right]}\lvert Z_{t}^{u_{\epsilon}}-\mathcal{G}^{0}\left(u_{\epsilon}\right)\left(t\right)\rvert^{2}\right)=0.
	$$

The verification of \textbf{(LDP)$\bm{_{1}}$} and \textbf{(LDP)$\bm{_{2}}$} will be given in two propositions below. Note that \textbf{(LDP)$\bm{_{2}}$} is stronger than (\textbf{LD}$_{\bm{2}}$) required in Theorem 2.11.
\end{proof}


\begin{proposition}\label{ld1}
	For any given $m\in\left(0,\infty\right)$, if $\left\{u_{n}, n\in \mathbb{N}\right\}\subset \mathcal{S}_{m}$ weakly converges to $u\in \mathcal{S}_{m}$ as $n\to \infty$, then
	$$
	\lim\limits_{n\to \infty}\sup\limits_{t\in \left[0,T\right]}\lvert \mathcal{G}^{0}\left(u_{n}\right)\left(t\right)-\mathcal{G}^{0}\left(u\right)\left(t\right) \rvert=0.
	$$
\end{proposition}	
	\begin{proof} According to the definition of $\mathcal{G}^{0}$, $Y^{u_{n}}=\mathcal{G}^{0}(u_{n})$ and $Y^{u}=\mathcal{G}^{0}(u)$. 
		
		Step1: $\left\{Y^{u_{n}}\right\}_{n\geq1}$ is pre-compact in $C\big([0,T],\overline{D(A)}\big)$. By Ascoli-Azela therom it suffices to prove that $\left\{Y^{u_{n}}\right\}_{n\geq1}$ is uniformly bounded and equi-continuous in $C\big([0,T],\overline{D(A)}\big)$. 

       The uniform boundedness is ensured by proposition \ref{Yu}. Let
        $$\sup\limits_{n\geq1}\sup\limits_{t\in[0,T]}\lvert Y_{t}^{u_{n}}\rvert=C_{\alpha},\quad \sup\limits_{n\geq1}\sup\limits_{t\in[0,T]}\lvert K_{t}^{u_{n}}\rvert=C_{\beta}.$$
	    
       Next we will prove that $\left\{Y^{u_{n}}\right\}_{n\geq1}$ is equi-continuous in $C\big([0,T],\overline{D(A)}\big)$.
	
	    For $T\geq t>s\geq0$,
	    \begin{equation}	Y_{t}^{u_{n}}-Y_{s}^{u_{n}}=\int_{s}^{t}b\left(Y_{r}^{u_{n}},\mathcal{L}_{X_{r}^{0}}\right)\,\mathrm{d}r+\int_{s}^{t}\sigma
\left(Y_{r}^{u_{n}},\mathcal{L}_{X_{r}^{0}}\right)u_{n}(r)\,\mathrm{d}r+|K^{u_{n}}|_{s}^{t}.
	    \end{equation}
        By (\hyperref[h1]{H1}) and Remark \ref{W2} we have
        \begin{align*}
        &\int_{s}^{t}\lvert b\left(Y_{r}^{u_{n}},\mathcal{L}_{X_{r}^{0}}\right)\rvert\mathrm{d}r \\
        \leq& \int_{s}^{t}\lvert b\left(Y_{r}^{u_{n}},\mathcal{L}_{X_{r}^{0}}\right)- b\left(0,\mathcal{L}_{X_{r}^{0}}\right)\rvert\mathrm{d}r +\int_{s}^{t}\lvert b\left(0,\mathcal{L}_{X_{r}^{0}}\right)- b\left(0,\delta_{0}\right)\rvert\mathrm{d}r+\int_{s}^{t}\lvert b\left(0,\delta_{0}\right)\rvert\mathrm{d}r\\
        \leq& \int_{s}^{t}L\left(1+\lvert Y_{r}^{u_{n}}\rvert^{q-1}\right)\lvert Y_{r}^{u_{n}}\rvert\,\mathrm{d}r+L\int_{s}^{t}|X_{r}^{0}|\,\mathrm{d}r+\int_{s}^{t}\lvert b\left(0,\delta_{0}\right)\rvert\mathrm{d}r \\
        \leq& L\left(C_{\alpha}\left(1+C_{\alpha}^{q-1}\right)+\sup\limits_{r\in[0,T]}|X_{r}^{0}|\right)|t-s|+\int_{s}^{t}\lvert b\left(0,\delta_{0}\right)\rvert\mathrm{d}r,
        \end{align*}
        and 
        \begin{align*}
        &\int_{s}^{t}\lvert \sigma\left(Y_{r}^{u_{n}},\mathcal{L}_{X_{r}^{0}}\right)u_{n}(r)\rvert\mathrm{d}r \\
        \leq& \int_{s}^{t}\lvert \sigma\left(Y_{r}^{u_{n}},\mathcal{L}_{X_{r}^{0}}\right)u_{n}(r)- \sigma\left(0,\delta_{0}\right)u_{n}(r)\rvert\mathrm{d}r+\int_{s}^{t}\lvert \sigma\left(0,\delta_{0}\right)u_{n}(r)\rvert\mathrm{d}r \\
        \leq& \int_{s}^{t}L\left(1+|Y_{r}^{u_{n}}|+|X_{r}^{0}|\right)|u_{n}(r)|\mathrm{d}r+\int_{s}^{t}\lVert\sigma\left(0,\delta_{0}\right)\rVert_{\ell_{2}}|u_{n}(r)|\mathrm{d}r \\
        \leq& L\left(1+C_{\alpha}+\sup _{r \in[0, T]}\left|X_{r}^{0}\right|\right)\left(\int_{0}^{T}\left|u_{n}(r)\right|^{2} \mathrm{~d} r\right)^{\frac{1}{2}}|t-s|^{\frac{1}{2}}+\left(\int_{s}^{t}\left|\sigma\left( 0, \delta_{0}\right)\right|^{2} \mathrm{~d} r\right)^{\frac{1}{2}}\left(\int_{0}^{T}\left|u_{n}(r)\right|^{2} \mathrm{~d} r\right)^{\frac{1}{2}}
        \end{align*}
        where H{\"o}lder inequality is used in the last inequality.
        
        By Remark \ref{r0}, there exists some $\rho(\epsilon)>0$ with $ t-s\leq\rho(\epsilon) $ such that
        $$\int_{s}^{t}\lvert b\left(0,\delta_{0}\right)\rvert\mathrm{d}r+m\left(\int_{s}^{t}\left|\sigma\left( 0, \delta_{0}\right)\right|^{2} \mathrm{~d} r\right)^{\frac{1}{2}}\leq \frac{\epsilon}{2}.$$
        Since $|K^{u_{n}}|$ is a continuous function with finite variation, once $ |t-s|<\frac{\epsilon}{2C_{\alpha,m,L,T}}$, we have
        $$
        	\lvert Y_{t}^{u_{n}}-Y_{s}^{u_{n}}\rvert \leq C_{\alpha,m,L,T}|t-s|+\frac{\epsilon}{2}\leq \epsilon.
        $$
        where $C_{\alpha,m,L,T}=min\left\{ L\left(C_{\alpha}\left(1+C_{\alpha}^{q-1}\right)+\sup\limits_{r\in[0,T]}|X_{r}^{0}|\right),Lm\left(1+C_{\alpha}+\sup\limits_{r \in[0, T]}\left|X_{r}^{0}\right|\right)\right\}$.

        This shows that $\left\{Y^{u_{n}}\right\}_{n\geq1}$ is pre-compact in $C\left([0,T],\mathbb{R}^{d}\right)$. \\

%
%

        Step2: Let $\eta$ and $\kappa$ be the limits of some subsequences of $\left\{Y^{u_{n}}\right\}_{n\geq 1}$ and $\left\{K^{u_{n}}\right\}_{n\geq 1}$ respectively. We will show that $\eta=Y^{u}$, $\kappa=K^{u}$. In that case, we assume that
        \begin{align}
        	\label{limy} \lim\limits_{n\to \infty}\sup_{t\in [0,T]}\lvert\eta(t)-Y^{u_{n}}(t)\rvert =0 \\
        	\label{limk} \lim\limits_{n\to \infty}\sup_{t\in [0,T]}\lvert\kappa(t)-K^{u_{n}}(t)\rvert =0
        \end{align}

        Notice that
        \begin{align}
        	\label{supy} \sup_{t\in [0,T]}\lvert \eta(t)\rvert\leq \sup_{n\geq 1}\sup_{t\in [0,T]}\lvert Y^{u_{n}}(t)\rvert=C_{\alpha}< \infty  \\
        	\label{supk} \sup_{t\in [0,T]}\lvert \kappa(t)\rvert\leq \sup_{n\geq 1}\sup_{t\in [0,T]}\lvert K^{u_{n}}(t)\rvert=C_{\beta}< \infty
        \end{align}
        By (\hyperref[h1]{H1}), (\ref{limy}) and (\ref{supy}), we have
        \begin{align*}
        	&\int_{0}^{T}\lvert b\left(Y^{u_{n}}(r),\mathcal{L}_{X^{0}_{r}}\right)-b\left(\eta(r),\mathcal{L}_{X^{0}_{r}}\right)\rvert\mathrm{d}r  \\
        	\leq L&\int_{0}^{T} \left(1+\lvert Y^{u_{n}}(r)\rvert^{q-1}+\lvert \eta(r)\rvert^{q-1}\right)\lvert Y^{u_{n}}(r)-\eta(r)\rvert\mathrm{d}r  \\
        	\leq L&\int_{0}^{T}\sup_{r\in [0,T]}\left(1+\lvert Y^{u_{n}}(r)\rvert^{q-1}+\lvert \eta(r)\rvert^{q-1}\right)\sup_{r\in [0,T]}\lvert Y^{u_{n}}(r)-\eta(r)\rvert\mathrm{d}r \\
        	\leq L&T\left(1+2C^{q-1}_{\alpha}\right)\sup_{r\in [0,T]}\lvert Y^{u_{n}}(r)-\eta(r)\rvert
        \end{align*}
        which implies that for each $t\in [0,T]$,
        \begin{align}\label{ldp1b}
        	\int_{0}^{t} b\left(Y^{u_{n}}(r),\mathcal{L}_{X^{0}_{r}}\right)\mathrm{d}r\, \to \int_{0}^{t} b\left(\eta(r),\mathcal{L}_{X^{0}_{r}}\right)\mathrm{d}r, \quad\text{as}\,\,\,n\to\infty.
        \end{align}
        By (\hyperref[h1]{H1}), (\ref{supy}), Remark \ref{W2}, and the fact that $X^0\in C\left([0,T],\mathbb{R}^{d}\right)$, we can get that
        \begin{align*}\label{}
        	\int_{0}^{T}\lVert\sigma\left(\eta(r),\mathcal{L}_{X_{r}^{0}}\right)\rVert\mathrm{~d}r<\infty.
        \end{align*}
        Recall that $u_{n}$ converges to $u$ weakly in $L^2\left([0,T],\mathbb{R}^{d}\right)$, for any $e\in\mathbb{R}^d$,
        \begin{align*}
        	\int_{0}^{t}\langle\sigma(\eta(r),\mathcal{L}_{X_{r}^0})u_{n}(r),e\rangle\mathrm{~d}r\to \int_{0}^{t}\langle\sigma(\eta(r),\mathcal{L}_{X_{r}^0})u(r),e\rangle\mathrm{~d}r,\quad \text{as}\,\,\,n\to\infty.
        \end{align*}
        Thus,
        \begin{align}\label{ldp1s}
        	\int_{0}^{t}\sigma(\eta(r),\mathcal{L}_{X_{r}^0})u_{n}(r)\mathrm{~d}r\to \int_{0}^{t}\sigma(\eta(r),\mathcal{L}_{X_{r}^0})u(r)\mathrm{~d}r,\quad \text{as}\,\,\,n\to\infty.
        \end{align}
        By (\hyperref[h1]{H1}), (\ref{limy}) and H{\"o}lder inequality, we obtain
        \begin{align*}
        	&\int_{0}^{T}\lvert\sigma(Y^{u_{n}}(r),\mathcal{L}_{X_{r}^0})u_{n}(r)-\sigma(\eta(r),\mathcal{L}_{X_{r}^0})u_{n}(r)\rvert\mathrm{~d}r   \\
        	\leq L&\int_{0}^{T}\lvert Y^{u_{n}}(r)-\eta(r)\rvert\lvert u_{n}(r)\rvert\mathrm{~d}r  \\
        	\leq L&\sup_{r\in[0,T]}\lvert Y^{u_{n}}(r)-\eta(r)\rvert \int_{0}^{T}\lvert u_{n}(r)\rvert\mathrm{~d}r   \\
        	\leq L&T^{\frac{1}{2}}\sup_{r\in[0,T]}\lvert Y^{u_{n}}(r)-\eta(r)\rvert\sup_{1\leq i\leq n}\left(\int_{0}^{T}\lvert u_{i}(r)\rvert^2\mathrm{~d}r\right)^{\frac{1}{2}}.
        \end{align*}
        Since $u_i\in \mathcal{S}_m$, we know that $\frac{1}{2}\int_{0}^{T}\lvert u_i(r)\rvert^2\mathrm{d}r\leq m$, then
        \begin{align}\label{ldp1e}
            \lim\limits_{n\to \infty}\int_{0}^{T}\lvert\sigma(Y^{u_{n}}(r),\mathcal{L}_{X_{r}^0})u_{n}(r)-\sigma(\eta(r),\mathcal{L}_{X_{r}^0})u_{n}(r)\rvert\mathrm{~d}r\to 0.
        \end{align}
        By (\ref{ldp1e}) and (\ref{ldp1s}) we have
        \begin{align}\label{ldp1c}
        \int_{0}^{t}\sigma(Y^{u_{n}}(r),\mathcal{L}_{X_{r}^0})u_{n}(r)\mathrm{~d}r\to \int_{0}^{t}\sigma(\eta(r),\mathcal{L}_{X_{r}^0})u(r)\mathrm{~d}r,\quad \text{as}\,\,\,n\to\infty.
        \end{align}

        Recall that $(Y^{u_n},K^{u_n})$ is the solution of (\ref{yy}) with $u$ replaced by $u_n$ which is the following equation,
        $$Y^{u_n}=h+\int_{0}^{t}b\left(Y_{s}^{u_{n}},\mathcal{L}_{X_{s}^{0}}\right)\mathrm{d}s+\int_{0}^{t}\sigma\left(Y_{s}^{u_{n}},\mathcal{L}_{X_{s}^{0}}\right)u_{n}(s)\mathrm{d}s-K^{u_n}_{t},\quad t\in[0,T].$$
        Combining the above (\ref{ldp1b}) and (\ref{ldp1c}) and let $n\to\infty$, we deduce that $(\eta,\kappa)$ is a solution to (\ref{yy}), and the uniqueness of the solution implies that $\eta=Y^{u_n}$ and $\kappa=K^{u_n}$, which completes the proof.

	\end{proof}


To verify \textbf{(LDP)$\bm{_{2}}$}, we need the following lemma.

\begin{lemma}\label{ep0}
	There exist two positive constant $\epsilon_{0}$ and $C_{T}$  such that
	$$\mathbb{E}\left(\sup_{t\in[0,T]}\lvert X_{t}^{\epsilon}-X_{t}^{0}\rvert^{2}\right)\leq C_{T}\left(\epsilon+\rho_{b,\epsilon}^{2}+\epsilon\rho_{\sigma,\epsilon}^{2}\right),\quad \forall\epsilon\in \left(0,\epsilon_{0}\right],$$
	where $\rho_{b,\epsilon}$ and $\rho_{\sigma,\epsilon}$ is given in \hyperref[h2]{\textbf{(H2)}}.
\end{lemma}

\begin{proof}
	By applying It$\hat{\text{o}}$ formula we have
	\begin{align*}
		&\lvert X_{t}^{\epsilon}-X_{t}^{0}\rvert^{2} \\
		=&2\int_{0}^{t}\left\langle b_{\epsilon}\left(X_{r}^{\epsilon},\mathcal{L}_{X_{r}^{\epsilon}}\right)-b\left(X_{r}^{0},\mathcal{L}_{X_{r}^{0}}\right), X_{r}^{\epsilon}-X_{r}^{0}\right\rangle\mathrm{d}r 
	+2\sqrt{\epsilon}\int_{0}^{t}\left\langle\sigma_{\epsilon}\left(X_{r}^{\epsilon},\mathcal{L}_{X_{r}^{\epsilon}}\right), X_{r}^{\epsilon}-X_{r}^{0}\right\rangle\mathrm{d}W(s) \\
		&+\epsilon\int_{0}^{t}\lVert\sigma_{\epsilon}\left(X_{r}^{\epsilon},\mathcal{L}_{X_{r}^{\epsilon}}\right)\rVert_{\mathcal{L}_{2}}^{2}\mathrm{d}r+\int_{0}^{t}\left\langle X_{r}^{\epsilon}-X_{r}^{0},\mathrm{d}K_{r}^{0}-\mathrm{d}K_{r}^{\epsilon}\right\rangle \\
		=:&J_{1}^{\epsilon}(t)+J_{2}^{\epsilon}(t)+J_{3}^{\epsilon}(t)+J_{4}^{\epsilon}(t).
	\end{align*}
	Next we estimate the four items in the last equality.
	
	For $J_{1}^{\epsilon}(t)$: \\
	By (\hyperref[h1]{\textbf{H1}}),(\hyperref[h2]{\textbf{H2}}) and Remark \ref{W2},
	\begin{align*}
		\sup_{t\in [0,T]}\lvert J_{1}^{\epsilon}(t)\rvert 
		\leq& \,2\int_{0}^{T}\lvert\left\langle X_{r}^{\epsilon}-X_{r}^{0},  b_{\epsilon}\left(X_{r}^{\epsilon},\mathcal{L}_{X_{r}^{\epsilon}}\right)-b\left(X_{r}^{\epsilon},\mathcal{L}_{X_{r}^{\epsilon}}\right)\right\rangle\rvert\mathrm{d}r \\
		&+2\int_{0}^{T}\lvert\left\langle X_{r}^{\epsilon}-X_{r}^{0},  b\left(X_{r}^{\epsilon},\mathcal{L}_{X_{r}^{\epsilon}}\right)-b\left(X_{r}^{0},\mathcal{L}_{X_{r}^{\epsilon}}\right)\right\rangle\rvert\mathrm{d}r \\
		&+2\int_{0}^{T}\lvert\left\langle X_{r}^{\epsilon}-X_{r}^{0},  b\left(X_{r}^{0},\mathcal{L}_{X_{r}^{\epsilon}}\right)-b\left(X_{r}^{0},\mathcal{L}_{X_{r}^{0}}\right)\right\rangle\rvert\mathrm{d}r \\
		\leq& \,2\rho_{b,\epsilon}\int_{0}^{T}\lvert X_{r}^{\epsilon}-X_{r}^{0}\rvert\mathrm{d}r+2L\int_{0}^{T}\lvert X_{r}^{\epsilon}-X_{r}^{0}\rvert^{2}\mathrm{d}r+2L\int_{0}^{T}\mathbb{W}^{2}\left(\mathcal{L}_{X_{r}^{\epsilon}},\mathcal{L}_{X_{r}^{0}}\right)\lvert X_{r}^{\epsilon}-X_{r}^{0}\rvert\mathrm{d}r \\
		\leq& \,(3L+1)\int_{0}^{T}\lvert X_{r}^{\epsilon}-X_{r}^{0}\rvert^{2}\,\mathrm{d}r+L\int_{0}^{T}
		\mathbb{W}^{2}\left(\mathcal{L}_{X_{r}^{\epsilon}},\mathcal{L}_{X_{r}^{0}}\right)\,\mathrm{d}r+\rho_{b,\epsilon}^{2}T \\
		\leq& \,(3L+1)\int_{0}^{T}\lvert X_{r}^{\epsilon}-X_{r}^{0}\rvert^{2}\,\mathrm{d}r+L\int_{0}^{T}
		\mathbb{E}\left(\lvert X_{r}^{\epsilon}-X_{r}^{0}\rvert^{2}\right)\,\mathrm{d}r+\rho_{b,\epsilon}^{2}T.\\
	\end{align*}
	Thus we have
	$$\mathbb{E}\left(\sup_{t\in [0,T]}\lvert J_{1}^{\epsilon}(t)\rvert\right)\leq (4L+1)\mathbb{E}\int_{0}^{T}\lvert X_{r}^{\epsilon}-X_{r}^{0}\rvert^{2}\,\mathrm{d}r+\rho_{b,\epsilon}^{2}T. $$
	
    For $J_{2}^{\epsilon}(t)$: \\
	By Burkholder-Davis-Gundy's inequality and Young's inequality, we have
	\begin{align*}
		\mathbb{E}\left(\sup_{t\in [0,T]}\lvert J_{2}^{\epsilon}(t)\rvert\right) 
		\leq& C\sqrt{\epsilon}\,\mathbb{E}\left[\int_{0}^{T}\lVert\sigma_{\epsilon}\left(X_{r}^{\epsilon},\mathcal{L}_{X_{r}^{\epsilon}}\right)\rVert_{\mathcal{L}_{2}}^{2}\cdot\lvert X_{r}^{\epsilon}-X_{r}^{0}\rvert^{2}\,\mathrm{d}r\right]^{\frac{1}{2}} \\
		\leq& \frac{1}{4}\mathbb{E}\left(\sup_{r\in[0,T]}\lvert X_{r}^{\epsilon}-X_{r}^{0}\rvert^{2}\right)+C\epsilon\,\mathbb{E}\int_{0}^{T}\lVert\sigma_{\epsilon}\left(X_{r}^{\epsilon},\mathcal{L}_{X_{r}^{\epsilon}}\right)\rVert_{\mathcal{L}_{2}}^{2}\,\mathrm{d}r \\
		\leq& \left(\frac{1}{4}+CL^{2}T\epsilon\right)\mathbb{E}\left(\sup_{r\in[0,T]}\lvert X_{r}^{\epsilon}-X_{r}^{0}\rvert^{2}\right)+C\epsilon\,\int_{0}^{T}\lVert\sigma\left(X_{r}^{\epsilon},\mathcal{L}_{X_{r}^{\epsilon}}\right)\rVert_{\mathcal{L}_{2}}^{2}\,\mathrm{d}r+C\epsilon\rho_{\sigma,\epsilon}^{2}T.
	\end{align*}

	For $J_{3}^{\epsilon}(t)$: \\
	By (\hyperref[h1]{H1}),(\hyperref[h2]{H2}) and Remark \ref{W2} again, we get
	\begin{align*}
		\mathbb{E}\left(\sup_{t\in [0,T]}\lvert J_{3}^{\epsilon}(t)\rvert\right) 
		=&\, \epsilon\,\mathbb{E}\int_{0}^{T}\lVert\sigma_{\epsilon}\left(X_{r}^{\epsilon},\mathcal{L}_{X_{r}^{\epsilon}}\right)\rVert_{\mathcal{L}_{2}}^{2}\,\mathrm{d}r \\
		\leq&\, C\epsilon\,\mathbb{E}\int_{0}^{T}\lVert\sigma_{\epsilon}\left(X_{r}^{\epsilon},\mathcal{L}_{X_{r}^{\epsilon}}\right)-\sigma\left(X_{r}^{\epsilon},\mathcal{L}_{X_{r}^{\epsilon}}\right)\rVert_{\mathcal{L}_{2}}^{2}\,\mathrm{d}r +C\epsilon\int_{0}^{T}\lVert\sigma\left(X_{r}^{0},\mathcal{L}_{X_{r}^{0}}\right)\rVert_{\mathcal{L}_{2}}^{2}\,\mathrm{d}r\\
		&+\, C\epsilon\,\mathbb{E}\int_{0}^{T}\lVert\sigma\left(X_{r}^{\epsilon},\mathcal{L}_{X_{r}^{\epsilon}}\right)-\sigma\left(X_{r}^{0},\mathcal{L}_{X_{r}^{0}}\right)\rVert_{\mathcal{L}_{2}}^{2}\,\mathrm{d}r \\
		\leq&\, C\epsilon\rho_{\sigma,\epsilon}^{2}T+CL^{2}\epsilon\,\mathbb{E}\int_{0}^{T}\left(\lvert X_{r}^{\epsilon}-X_{r}^{0}\rvert^{2}+\mathbb{W}^{2}\left(\mathcal{L}_{X_{r}^{\epsilon}},\mathcal{L}_{X_{r}^{0}}\right)\right)\mathrm{d}r \\
		&+\,C\epsilon\int_{0}^{T}\lVert\sigma\left(X_{r}^{0},\mathcal{L}_{X_{r}^{0}}\right)\rVert_{\mathcal{L}_{2}}^{2}\,\mathrm{d}r \\
		\leq&\, C\epsilon\rho_{\sigma,\epsilon}^{2}T+CL^{2}T\epsilon\,\mathbb{E}\left(\sup_{r\in[0,T]}\lvert X_{r}^{\epsilon}-X_{r}^{0}\rvert^{2}\right)+\,C\epsilon\int_{0}^{T}\lVert\sigma\left(X_{r}^{0},\mathcal{L}_{X_{r}^{0}}\right)\rVert_{\mathcal{L}_{2}}^{2}\,\mathrm{d}r, \\
	\end{align*}
	
	For $J_{4}^{\epsilon}(t)$: \\
	From (\ref{mono}) we can easily get
	$$\sup_{t\in [0,T]}\lvert J_{4}^{\epsilon}(t)\rvert\leq 0.$$

	Combining the four estimates together, we obtain
	\begin{align*}
		&\left(\frac{3}{4}-CL^{2}T\epsilon\right)\mathbb{E}\left(\sup_{r\in[0,T]}\lvert X_{r}^{\epsilon}-X_{r}^{0}\rvert^{2}\right) \\
		\leq& (4L+1)\mathbb{E}\int_{0}^{T}\lvert X_{r}^{\epsilon}-X_{r}^{0}\rvert^{2}\,\mathrm{d}r+\rho_{b,\epsilon}^{2}T+C\epsilon\,\int_{0}^{T}\lVert\sigma\left(X_{r}^{\epsilon},\mathcal{L}_{X_{r}^{\epsilon}}\right)\rVert_{\mathcal{L}_{2}}^{2}\,\mathrm{d}r+C\epsilon\rho_{\sigma,\epsilon}^{2}T.
	\end{align*}
	
	Since $X^{0}\in C\left([0,T],\mathbb{R}^{d}\right)$ and  by (\hyperref[h1]{H1}),  there exists $\epsilon_{0}>0$ small enough such that for any $\epsilon\in \left(0,\epsilon_{0}\right]$,
		$$\frac{3}{4}-CL^{2}T\epsilon\geq\frac{1}{4},$$
	Finally by Gronwall's inequality, there exists a constant $C_{T}>0$ such that for any $\epsilon\in \left(0,\epsilon_{0}\right]$,
	$$\mathbb{E}\left(\sup_{t\in[0,T]}\lvert X_{t}^{\epsilon}-X_{t}^{0}\rvert^{2}\right)\leq C_{T}\left(\epsilon+\rho_{b,\epsilon}^{2}+\epsilon\rho_{\sigma,\epsilon}^{2}\right),\quad \forall\epsilon\in \left(0,\epsilon_{0}\right].$$
\end{proof}

Now we are in the position to verify  \textbf{(LDP)$\bm{_{2}}$}.

\begin{proposition}\label{ld2}
	For any given $m\in\left(0,\infty\right)$, let $\left\{u_{\epsilon}, \epsilon>0\right\}\subset \mathcal{S}_{m}$. Then
	$$
	\lim\limits_{\epsilon\to 0}\mathbb{E}\left(\sup\limits_{t\in \left[0,T\right]}\lvert Z_{t}^{u_{\epsilon}}-\mathcal{G}^{0}\left(u_{\epsilon}\right)\left(t\right)\rvert^{2}\right)=0.
	$$
	\begin{proof}
		Let $Y_{t}^{u_{\epsilon}}$ be the solution of (\ref{yy}) with $u$ replaced by $u_{\epsilon}$, so $\mathcal{G}^{0}(u_{\epsilon})=Y^{u_{\epsilon}}$. Then
		\begin{align*}
			\phi_{t}^{\epsilon,u_{\epsilon}}:&=Z_{t}^{u_{\epsilon}}-Y_{t}^{u_{\epsilon}}\\
=&K_{t}^{u_{\epsilon}}-K_{t}^{\epsilon,u_{\epsilon}} +\int_{0}^{t}\left[b_{\epsilon}\left(Z_{r}^{u_{\epsilon}},\mathcal{L}_{X_{r}^{\epsilon}}\right)-b\left(Y_{r}^{u_{\epsilon}},\mathcal{L}_{X_{r}^{0}}\right)\right]\,\mathrm{d}r \\
			&+\int_{0}^{t}\left[\sigma_{\epsilon}\left(Z_{r}^{u_{\epsilon}},\mathcal{L}_{X_{r}^{\epsilon}}\right)u_{\epsilon}(r)-\sigma\left(Y_{r}^{u_{\epsilon}},\mathcal{L}_{X_{r}^{0}}\right)u_{\epsilon}(r)\right]\,\mathrm{d}r +\sqrt{\epsilon}\int_{0}^{t}\sigma_{\epsilon}\left(Z_{r}^{u_{\epsilon}},\mathcal{L}_{X_{r}^{\epsilon}}\right)\,\mathrm{d}W(r).
		\end{align*}
		By It$\hat{\text{o}}$'s formula we have
		\begin{align*}
			\lvert\phi_{t}^{\epsilon,u_{\epsilon}}\rvert^{2}=&2\int_{0}^{t}\left\langle \phi_{r}^{\epsilon,u_{\epsilon}},\mathrm{d}K_{r}^{u_{\epsilon}}-\mathrm{d}K_{r}^{\epsilon,u_{\epsilon}}\right\rangle +2\int_{0}^{t}\left\langle\phi_{r}^{\epsilon,u_{\epsilon}},b_{\epsilon}\left(Z_{r}^{u_{\epsilon}},\mathcal{L}_{X_{r}^{\epsilon}}\right)-b\left(Y_{r}^{u_{\epsilon}},\mathcal{L}_{X_{r}^{0}}\right)\right\rangle\,\mathrm{d}r \\
			&+2\int_{0}^{t}\left\langle\phi_{r}^{\epsilon,u_{\epsilon}},\sigma_{\epsilon}\left(Z_{r}^{u_{\epsilon}},\mathcal{L}_{X_{r}^{\epsilon}}\right)u_{\epsilon}(r)-\sigma\left(Y_{r}^{u_{\epsilon}},\mathcal{L}_{X_{r}^{0}}\right)u_{\epsilon}(r)\right\rangle\,\mathrm{d}r \\
			&+2\sqrt{\epsilon}\int_{0}^{t}\left\langle\phi_{r}^{\epsilon,u_{\epsilon}},\sigma_{\epsilon}\left(Z_{r}^{u_{\epsilon}},\mathcal{L}_{X_{r}^{\epsilon}}\right)\,\mathrm{d}W(r)\right\rangle +\epsilon\int_{0}^{t}\lVert \sigma_{\epsilon}\left(Z_{r}^{u_{\epsilon}},\mathcal{L}_{X_{r}^{\epsilon}}\right)\rVert_{\mathcal{L}_{2}}^{2}\mathrm{d}r \\
			=:&I_{1}^{\epsilon}(t)+I_{2}^{\epsilon}(t)+I_{3}^{\epsilon}(t)+I_{4}^{\epsilon}(t)+I_{5}^{\epsilon}(t).
		\end{align*}
	Next we estimate the five items in the equality above.
	    
For $I_{1}^{\epsilon}(t)$: \\
	    By (\ref{mono}) we get immediately 
	    $\lvert I_{1}^{\epsilon}(t)\rvert\leq 0.$
	
	    For $I_{2}^{\epsilon}(t)$: \\
	    By Assumptions (\hyperref[h1]{H1}),(\hyperref[h2]{H2}) and Remark \ref{W2}, we have
	    \begin{align*}
	    	\sup_{t\in [0,T]}\lvert I_{2}^{\epsilon}(t)\rvert 
	    	\leq& \,2\int_{0}^{T}\lvert\left\langle \phi_{r}^{\epsilon,u_{\epsilon}},  b_{\epsilon}\left(Z_{r}^{u_{\epsilon}},\mathcal{L}_{X_{r}^{\epsilon}}\right)-b\left(Y_{r}^{u_{\epsilon}},\mathcal{L}_{X_{r}^{0}}\right)\right\rangle\rvert\mathrm{d}r \\
	    	\leq& \,2\int_{0}^{T}\lvert\left\langle \phi_{r}^{\epsilon,u_{\epsilon}},  b_{\epsilon}\left(Z_{r}^{u_{\epsilon}},\mathcal{L}_{X_{r}^{\epsilon}}\right)-b\left(Z_{r}^{u_{\epsilon}},\mathcal{L}_{X_{r}^{\epsilon}}\right)\right\rangle\rvert\mathrm{d}r \\
	    	&+2\int_{0}^{T}\lvert\left\langle \phi_{r}^{\epsilon,u_{\epsilon}},  b\left(Z_{r}^{u_{\epsilon}},\mathcal{L}_{X_{r}^{\epsilon}}\right)-b\left(Y_{r}^{u_{\epsilon}},\mathcal{L}_{X_{r}^{\epsilon}}\right)\right\rangle\rvert\mathrm{d}r \\
	    	&+2\int_{0}^{T}\lvert\left\langle \phi_{r}^{\epsilon,u_{\epsilon}},  b\left(Y_{r}^{u_{\epsilon}},\mathcal{L}_{X_{r}^{\epsilon}}\right)-b\left(Y_{r}^{u_{\epsilon}},\mathcal{L}_{X_{r}^{0}}\right)\right\rangle\rvert\mathrm{d}r \\
	    	\leq&\,2\rho_{b,\epsilon}\int_{0}^{T}\lvert\phi_{r}^{\epsilon,u_{\epsilon}}\rvert\mathrm{d}r+2L\int_{0}^{T}\lvert\phi_{r}^{\epsilon,u_{\epsilon}}\rvert^{2}\mathrm{d}r+2L\int_{0}^{T}\mathbb{W}^{2}\left(\mathcal{L}_{X_{r}^{\epsilon}},\mathcal{L}_{X_{r}^{0}}\right)\lvert \phi_{r}^{\epsilon, u_{\epsilon}}\rvert\mathrm{d}r \\
	    	\leq&\left[\int_{0}^{T}\rho_{b,\epsilon}^{2}\mathrm{d}r+\int_{0}^{T}\lvert\phi_{r}^{\epsilon,u_{\epsilon}}\rvert^{2}\mathrm{d}r\right]+2L\int_{0}^{T}\lvert\phi_{r}^{\epsilon,u_{\epsilon}}\rvert^{2}\mathrm{d}r+2L\int_{0}^{T}\left[\mathbb{E}\left(\lvert X_{r}^{\epsilon}-X_{r}^{0}\rvert^{2}\right)\right]^{\frac{1}{2}}\lvert \phi_{r}^{\epsilon,u_{\epsilon}}\rvert\mathrm{d}r \\
	    	\leq&\,\rho_{b,\epsilon}^{2}T+\left(2L+1\right)\int_{0}^{T}\lvert\phi_{r}^{\epsilon,u_{\epsilon}}\rvert^{2}\mathrm{d}r+L\left[\int_{0}^{T}\mathbb{E}\left(\lvert X_{r}^{\epsilon}-X_{r}^{0}\rvert^{2}\right)\mathrm{d}r+\int_{0}^{T}\lvert\phi_{r}^{\epsilon,u_{\epsilon}}\rvert^{2}\mathrm{d}r\right] \\
	    	\leq&\left(3L+1\right)\int_{0}^{T}\lvert\phi_{r}^{\epsilon,u_{\epsilon}}\rvert^{2}\mathrm{d}r+LT\mathbb{E}\left(\sup_{r\in[0,T]}\lvert X_{r}^{\epsilon}-X_{r}^{0}\rvert^{2}\right)+\rho_{b,\epsilon}^{2}T.	
	    \end{align*}
	
	    Hence $$\mathbb{E}\left(\sup_{r \in[0, T]}\lvert I_{2}^{\epsilon}(t)\rvert\right)\leq \left(3L+1\right)\int_{0}^{T}\mathbb{E}\left(\sup_{s\in[0,r]}\lvert\phi_{s}^{\epsilon,u_{\epsilon}}\rvert^{2}\right)\mathrm{d}r+LT\mathbb{E}\left(\sup_{r\in[0,T]}\lvert X_{r}^{\epsilon}-X_{r}^{0}\rvert^{2}\right)+\rho_{b,\epsilon}^{2}T.$$
	
	    For $I_{3}^{\epsilon}(t)$: \\
	    For any $\tau>0$ we have
	    \begin{align*}
	    	&\mathbb{E}\left(\sup_{t\in [0,T]}\lvert I_{3}^{\epsilon}(t)\rvert\right) \\
	    	\leq& \, 2\mathbb{E}\int_{0}^{T}\lvert\left\langle \phi_{r}^{\epsilon,u_{\epsilon}}, \left[\sigma_{\epsilon}\left(Z_{r}^{u_{\epsilon}},\mathcal{L}_{X_{r}^{\epsilon}}\right)-\sigma\left(Z_{r}^{u_{\epsilon}},\mathcal{L}_{X_{r}^{0}}\right)\right]u_{\epsilon}(r)\right\rangle\rvert\,\mathrm{d}r \\
	    	&+ 2\mathbb{E}\int_{0}^{T}\lvert\left\langle \phi_{r}^{\epsilon,u_{\epsilon}}, \left[\sigma\left(Z_{r}^{u_{\epsilon}},\mathcal{L}_{X_{r}^{\epsilon}}\right)-\sigma\left(Y_{r}^{u_{\epsilon}},\mathcal{L}_{X_{r}^{0}}\right)\right]u_{\epsilon}(r)\right\rangle\rvert\,\mathrm{d}r \\
	    	\leq& \,2\rho_{\sigma,\epsilon}\mathbb{E}\int_{0}^{T}\lvert\phi_{r}^{\epsilon,u_{\epsilon}}\rvert\lvert u_{\epsilon}(r)\rvert\,\mathrm{d}r+2\mathbb{E}\int_{0}^{T}\lvert\phi_{r}^{\epsilon,u_{\epsilon}}\rvert\cdot\lVert\sigma\left(Z_{r}^{u_{\epsilon}},\mathcal{L}_{X_{r}^{\epsilon}}\right)-\sigma\left(Y_{r}^{u_{\epsilon}},\mathcal{L}_{X_{r}^{0}}\right)\rVert_{\mathcal{L}_{2}}\cdot\lvert u_{\epsilon}(r)\rvert\,\mathrm{d}r \\
	    	\leq& \, \int_{0}^{T}\mathbb{E}\left(\sup_{s\in [0,r]}\lvert\phi_{s}^{\epsilon,u_{\epsilon}}\rvert^{2}\right)\,\mathrm{d}r+\rho_{\sigma,\epsilon}^{2}\int_{0}^{T}\lvert u_{\epsilon}(r)\rvert^{2}\,\mathrm{d}r \\
	    	&+2\mathbb{E}\left(\sup_{r\in[0,T]}\lvert\phi_{r}^{\epsilon,u_{\epsilon}}\rvert\left(\int_{0}^{T}\lVert\sigma\left(Z_{r}^{u_{\epsilon}},\mathcal{L}_{X_{r}^{\epsilon}}\right)-\sigma\left(Y_{r}^{u_{\epsilon}},\mathcal{L}_{X_{r}^{0}}\right)\rVert_{\mathcal{L}_{2}}^{2}\,\mathrm{d}r\right)^{\frac{1}{2}}\left(\int_{0}^{T}\lvert u_{\epsilon}(r)\rvert^{2}\,\mathrm{d}r\right)^{\frac{1}{2}}\right) \\
	    	\leq& \, \int_{0}^{T}\mathbb{E}\left(\sup_{s\in [0,r]}\lvert\phi_{s}^{\epsilon,u_{\epsilon}}\rvert^{2}\right)\,\mathrm{d}r+\rho_{\sigma,\epsilon}^{2}m+\tau\mathbb{E}\left(\sup_{r\in[0,T]}\lvert\phi_{r}^{\epsilon,u_{\epsilon}}\rvert^{2}\right)+C_{m,T,\tau,L_{\sigma}}\int_{0}^{T}\left(\mathbb{E}\lvert\phi_{r}^{\epsilon,u_{\epsilon}}\rvert^{2}+\mathbb{E}\lvert X_{r}^{\epsilon}-X_{r}^{0}\rvert^{2}\right)\,\mathrm{d}r \\
	    	=&\tau\mathbb{E}\left(\sup_{r\in[0,T]}\lvert\phi_{r}^{\epsilon,u_{\epsilon}}\rvert^{2}\right)+C_{m,T,\tau,L_{\sigma}}\int_{0}^{T}\mathbb{E}\left(\sup_{s\in [0,r]}\lvert\phi_{s}^{\epsilon,u_{\epsilon}}\rvert^{2}\right)\,\mathrm{d}r+C_{m,T,\tau,L_{\sigma}}\mathbb{E}\left(\sup_{r\in[0,T]}\lvert X_{r}^{\epsilon}-X_{r}^{0}\rvert^{2}\right)+\rho_{\sigma,\epsilon}^{2}m.
	   \end{align*}
	
	    For $I_{4}^{\epsilon}(t)$ and $I_{5}^{\epsilon}(t)$: \\
	    \begin{align*}
	    	&\mathbb{E}\left(\sup_{t\in [0,T]}\lvert I_{4}^{\epsilon}(t)\rvert\right)+\mathbb{E}\left(\sup_{t\in [0,T]}\lvert I_{5}^{\epsilon}(t)\rvert\right) \\
	    	\leq&\, C\sqrt{\epsilon}\,\mathbb{E}\left[\int_{0}^{T}\lVert\sigma_{\epsilon}\left(Z_{r}^{u_{\epsilon}},\mathcal{L}_{X_{r}^{\epsilon}}\right)\rVert_{\mathcal{L}_{2}}^{2}\cdot\lvert\phi_{r}^{\epsilon,u_{\epsilon}}\rvert^{2}\,\mathrm{d}r\right]^{\frac{1}{2}}+\epsilon\,\mathbb{E}\int_{0}^{T}\lVert\sigma_{\epsilon}\left(Z_{r}^{u_{\epsilon}},\mathcal{L}_{X_{r}^{\epsilon}}\right)\rVert_{\mathcal{L}_{2}}^{2}\,\mathrm{d}r \\
	    	\leq& \,\tau\mathbb{E}\left(\sup_{r\in [0,T]}\lvert\phi_{r}^{\epsilon,u_{\epsilon}}\rvert^{2}\right)+C_{\tau}\epsilon\int_{0}^{T}\lVert\sigma_{\epsilon}\left(Z_{r}^{u_{\epsilon}},\mathcal{L}_{X_{r}^{\epsilon}}\right)\rVert_{\mathcal{L}_{2}}^{2}\,\mathrm{d}r \\
	    	\leq& \, \tau\mathbb{E}\left(\sup_{r\in [0,T]}\lvert\phi_{r}^{\epsilon,u_{\epsilon}}\rvert^{2}\right)+C_{\tau}\epsilon\,\mathbb{E}\int_{0}^{T}\lVert\sigma_{\epsilon}\left(Z_{r}^{u_{\epsilon}},\mathcal{L}_{X_{r}^{\epsilon}}\right)-\sigma\left(Z_{r}^{u_{\epsilon}},\mathcal{L}_{X_{r}^{\epsilon}}\right)\rVert_{\mathcal{L}_{2}}^{2}\,\mathrm{d}r \\
	    	&+C_{\tau}\epsilon\,\mathbb{E}\int_{0}^{T}\lVert\sigma\left(Z_{r}^{u_{\epsilon}},\mathcal{L}_{X_{r}^{\epsilon}}\right)-\sigma\left(Y_{r}^{u_{\epsilon}},\mathcal{L}_{X_{r}^{0}}\right)\rVert_{\mathcal{L}_{2}}^{2}\,\mathrm{d}r \\
	    	&+C_{\tau}\epsilon\,\mathbb{E}\int_{0}^{T}\lVert\sigma\left(Y_{r}^{u_{\epsilon}},\mathcal{L}_{X_{r}^{0}}\right)-\sigma\left(0,\delta_{0}\right)\rVert_{\mathcal{L}_{2}}^{2}\,\mathrm{d}r+C_{\tau}\epsilon\int_{0}^{T}\lVert\sigma\left(0,\delta_{0}\right)\rVert_{\mathcal{L}_{2}}^{2}\,\mathrm{d}r \\
	    	\leq& \, \tau\mathbb{E}\left(\sup_{r\in [0,T]}\lvert\phi_{r}^{\epsilon,u_{\epsilon}}\rvert^{2}\right)+C_{\tau}\rho_{\sigma,\epsilon}^{2}T\epsilon+C_{\tau,L_{\sigma}}\epsilon\,\int_{0}^{T}\mathbb{E}\left(\sup_{r\in [0,T]}\lvert\phi_{r}^{\epsilon,u_{\epsilon}}\rvert^{2}\right)\,\mathrm{d}r \\
	    	&+ C_{\tau,L_{\sigma}}\epsilon\,\int_{0}^{T}\mathbb{E}\lvert X_{r}^{\epsilon}-X_{r}^{0}\rvert^{2}\,\mathrm{d}r+C_{\tau}T\epsilon\sup_{u_{\epsilon}\in \mathcal{S}_{m}}\sup_{r\in [0,T]}\lvert Y_{r}^{u_{\epsilon}}\rvert^{2}+C_{\tau}T\epsilon\sup_{r\in [0,T]}|X_{r}^{0}|^{2}+C_{\tau,\sigma_{0}}\epsilon \\
	    	\leq&\, \left(\tau+C\epsilon\right)\mathbb{E}\left(\sup_{r\in [0,T]}\lvert\phi_{r}^{\epsilon,u_{\epsilon}}\rvert^{2}\right)+C\epsilon+C\epsilon\,\mathbb{E}\left(\sup_{r\in[0,T]}\lvert X_{r}^{\epsilon}-X_{r}^{0}\rvert^{2}\right)+C\rho_{\sigma,\epsilon}^{2}T\epsilon,
	    \end{align*}
		where $C$ is dependent on $\tau,L_{\sigma},L_{b},T,Y^{u_\epsilon},X^{0},\sigma_{0},\epsilon_{0}$.
		
Taking $\tau=\frac{1}{10}$ and combining the five estimates above, we get for any $\epsilon\in \left(0,\epsilon_{0}\right]$,
		\begin{align*}
			&\left(\frac{4}{5}-C\epsilon\right)\mathbb{E}\left(\sup_{r\in [0,T]}\lvert\phi_{r}^{\epsilon,u_{\epsilon}}\rvert^{2}\right) \leq\, C\int_{0}^{T}\mathbb{E}\left(\sup_{s\in[0,r]}\lvert\phi_{s}^{\epsilon,u_{\epsilon}}\rvert^{2}\right)\mathrm{d}r+C\epsilon+C\left(\rho_{b,\epsilon}^{2}+\rho_{\sigma,\epsilon}^{2}+\epsilon\rho_{\sigma,\epsilon}^{2}\right),
		\end{align*}
		Since $\rho_{\sigma,\epsilon}$ and $\rho_{b,\epsilon}$ convergence to $0$ as $\epsilon\to 0$,
		we obtain by the Gronwall's inequality 
		$$\lim\limits_{\epsilon\to 0}\mathbb{E}\left(\sup_{t\in[0,T]}\lvert Z_{t}^{u_{\epsilon}}-Y_{t}^{u_{\epsilon}}\rvert^{2}\right)=0,$$
		which is the desired result.
		
	\end{proof}
\end{proposition}

\subsection{Moderate deviation principle}
        Lemma \ref{girs} can be also used to establish the MDP of $X^\epsilon$ as $\epsilon$ decrease to 0. We introduce the following conditions.

        Assume $\lambda(\epsilon)>0,\,\, \epsilon>0$ satisfies
        \begin{align}\label{lame}
        	\lambda(\epsilon)\to 0, \quad \frac{\epsilon}{\lambda^2(\epsilon)}\to 0, \quad \text{as}\,\,\epsilon\to 0.
        \end{align}

        Define
        \begin{align*}
        	M^{\epsilon}_t:=\frac{1}{\lambda(\epsilon)}(X_t^\epsilon-X_t^0), \quad t\in[0,T].
        \end{align*}
        Then $M^{\epsilon}_t$ satisfies the following multivalued SDE
        \begin{align}
        	\mathrm{d}M^{\epsilon}_t\in&\,\,\frac{1}{\lambda(\epsilon)}\int_{0}^{t}\left(b_{\epsilon}(\lambda(\epsilon)M^{\epsilon}_s+X^0_s,\mathcal{L}_{X_s^{\epsilon}})-b(X_s^0,\mathcal{L}_{X_s^0})\right)\mathrm{d}s  \nonumber\\
        	&+\frac{\sqrt{\epsilon}}{\lambda(\epsilon)}\int_{0}^{t}\sigma_{\epsilon}(\lambda(\epsilon)M^{\epsilon}_s+X_s^0,\mathcal{L}_{X_s^{\epsilon}})\mathrm{d}W_s-\text{A}(	M^{\epsilon}_t)\mathrm{d}t.
        \end{align}
        By Proposition \ref{solu} and Proposition \ref{strongsolu}, $(M^{\epsilon}_t,\hat{K}^\epsilon)$ is the unique solution to the following equation
        \begin{align}\label{mep}
        	\mathrm{d}M^{\epsilon}_t=&\frac{1}{\lambda(\epsilon)}\int_{0}^{t}\left(b_{\epsilon}(\lambda(\epsilon)M^{\epsilon}_s+X^0_s,\mathcal{L}_{X_s^{\epsilon}})-b(X_s^0,\mathcal{L}_{X_s^0})\right)\mathrm{d}s   \nonumber \\
        	&+\frac{\sqrt{\epsilon}}{\lambda(\epsilon)}\int_{0}^{t}\sigma_{\epsilon}(\lambda(\epsilon)M^{\epsilon}_s+X_s^0,\mathcal{L}_{X_s^{\epsilon}})\mathrm{d}W_s-\mathrm{d}\hat{K}^\epsilon_t.
        \end{align}

        Let
        \begin{align*}	\Gamma^\epsilon_{\mathcal{L}_{X^\epsilon}}(\cdot):=\frac{1}{\lambda(\epsilon)}\left(\mathcal{G}^\epsilon_{\mathcal{L}_{X^\epsilon}}(\cdot)-X^0\right),
        \end{align*}
        then
        \begin{enumerate}
        	\item[(a)] $\Gamma^\epsilon_{\mathcal{L}_{X^\epsilon}}$ is a measurable map from $C([0,T],\mathbb{R}^d)\mapsto C([0,T],\overline{D(A)}) $ such that
        	$$M^\epsilon=\Gamma^\epsilon_{\mathcal{L}_{X^\epsilon}}\left(\sqrt{\epsilon}W(\cdot)\right)$$
        	\item[(b)] For any $m\in(0,\infty)$, $\psi_\epsilon\in\mathcal{D}_m^T$, let
        	\begin{align}\label{mpsi}
        		M^{\psi_\epsilon}:=\Gamma^\epsilon_{\mathcal{L}_{X^\epsilon}}\left(\sqrt{\epsilon}W(\cdot)+\lambda(\epsilon)\int_{0}^{\cdot}\psi_\epsilon\mathrm{d}s\right).
        	\end{align}
        \end{enumerate}

        We make the following assumptions:

        \textbf{(B0)} There exist $L',q'\geq 0$ such that for all $x,x'\in\mathbb{R}^{d}$,
        \begin{align}\label{b0}
        |b'(x,\mathcal{L}_{X^0_s})-b'(x',\mathcal{L}_{x^0_s})|\leq L'(1+|x|^{q'}+|x'|^{q'})|x-x'|.
        \end{align}

        \textbf{(B1)}
        \begin{align}\label{b1}
        	\int_{0}^{T}|b'(X^0_t,\mathcal{L}_{X^0_t})|\mathrm{d}t<\infty.
        \end{align}
        where $b'(x,\mu)$ denotes the derivative of $b(x,\mu)$ with respect to the variable $x$.
        
            \textbf{(B2)}

    \begin{equation}\label{b2}
    	\begin{aligned}
    		\lim\limits_{\epsilon\to 0}\frac{\rho_{b,\epsilon}}{\lambda(\epsilon)}=0, 
    	\end{aligned}
    \end{equation}
    where $\rho_{b,\epsilon}$  is given in (\hyperref[h2]{\text{H2}}).


        \begin{proposition}\label{nu}
        	Assume that (\hyperref[h1]{H1}),(\hyperref[b0]{B0}) and (\hyperref[b1]{B1}) hold. Then for any fixed $m\in(0,\infty)$ and $\psi\in \mathcal{S}_m$, there is a unique solution $\nu^\psi=\big\{\big(\nu^{\psi}(t),\hat{K}^\psi(t)\big), t\in[0,T]\big\}\in C([0,T],\mathbb{R}^d)$ to the following equation:
        	\begin{equation}\label{mdp1}
        		\left\{ \begin{aligned}
        		&\mathrm{d}\nu^{\psi}(t)=b'(X^0_t,\mathcal{L}_{X^0_t})\nu^{\psi}(t)\mathrm{d}t+\sigma(X_t^0,\mathcal{L}_{X_t^0})\psi(t)\mathrm{d}t-\mathrm{d}\hat{K}^\psi_t     \\
        		&\nu^{\psi}(0)=0.
        		\end{aligned}  \right.
        	\end{equation}
        	Moreover,
        	\begin{align}\label{nulim}
        		\sup_{\psi\in \mathcal{S}^{m}}\sup_{t\in[0,T]}|\nu^{\psi}(t)|<\infty.
        	\end{align}
        	
        	\begin{proof}
        		Since $\big\{\hat{K}^\psi(t), t\in[0,T]\big\}$ is of finite variation with $\psi\in\mathcal{S}_m$, we have
        		$|\hat{K}^\psi|_0^T<\infty.$
        		
        		By (\hyperref[h1]{H1}), Remark \ref{W2} and the fact that $X^0\in C([0,T],\mathbb{R}^d)$ and $\psi\in\mathcal{S}_m$, we have
        		\begin{align}\label{slim}
        			\int_{0}^{T}\lVert \sigma(X_s^0,\mathcal{L}_{X_s^0})\rVert^2\mathrm{d}s<\infty.
        		\end{align}
        		Thus
        		\begin{align}\label{si}
        			\int_{0}^{T}\lvert\sigma(X_t^0,\mathcal{L}_{X_t^0})\psi(t)\rvert\mathrm{d}t  \nonumber
        			\leq &\left(\int_{0}^{T}\lVert\sigma(X_t^0,\mathcal{L}_{X_t^0})\rVert^2\mathrm{d}t\right)^{\frac{1}{2}}\left(\int_{0}^{T}|\psi(t)|^2\mathrm{d}t\right)^\frac{1}{2}   \nonumber\\
        			\leq &\left(\int_{0}^{T}\lvert\sigma(X_t^0,\mathcal{L}_{X_t^0})\rvert^2\mathrm{d}t\right)^{\frac{1}{2}}(2m)^{\frac{1}{2}}  \nonumber
        			<\infty.
        		\end{align}
        	
        		Due to (\hyperref[b1]{B1}) and the estimates above we can easily prove that the linear equation (\ref{mdp1}) has a unique solution $\big\{\big(\nu^{\psi}(t),\hat{K}^\psi(t)\big), t\in[0,T]\big\}$.
        		By using Gronwall's inequality we have
        		\begin{align}
        		\nu^{\psi}(t)\leq  e^{\int_{0}^{t}b'(X_s^0,\mathcal{L}_{X_s^0})\mathrm{d}s}\int_{0}^{t}\sigma(X_s^0,\mathcal{L}_{X_s^0})\psi(s)\mathrm{d}s,
        		\end{align}
        		which implies (\ref{nulim}).
        		
        	\end{proof}

        \end{proposition}

    Now we state our main result about the moderate deviation principle.
    \begin{theorem}
    	Assume that (\hyperref[h1]{H1}), (\hyperref[h2]{H2}), (H3), (H4), (\hyperref[b0]{B0}), (\hyperref[b1]{B1}) and (\hyperref[b2]{B2}) hold, then $\left\{M^{\epsilon},\epsilon>0\right\}$ satisfies a LDP on $C([0,T],\overline{D(A)})$ with speed $\frac{\epsilon}{\lambda^2(\epsilon)}$ and the rate function $I$ given by
    	\begin{align}
    		I(g):=\frac{1}{2}\inf_{\left\{\psi\in L^2([0,T],\mathbb{R}^d), \nu^{\psi}=g\right\}}\int_{0}^{T}|\psi(s)|^2\mathrm{d}s, \quad g\in C\left([0,T],\overline{D(A)}\right),
    	\end{align}
    	where for $\psi\in L^2([0,T],\mathbb{R}^d)$, $\big(\nu^{\psi},\hat{K}^\psi\big)$ is the unique solution of (\ref{mdp1}). 
    \end{theorem}

    \begin{proof}
    	By proposition \ref{nu}, we can define a map
    	\begin{align}
    		\Gamma^0:L^2([0,T],\mathbb{R}^d)\ni \psi \mapsto \nu^\psi\in C([0,T],\overline{D(A)}),
    	\end{align}
    	where $\nu^\psi$ is the unique solution of (\ref{nu}).
    	
    	For any $\epsilon>0, m\in(0,\infty)$, and $\psi_{\epsilon}\in\mathcal{S}_{m}$, recall that $\left\{\left(M^{\psi_\epsilon}(t),\hat{K}^{\epsilon,\psi_\epsilon}(t)\right),t\in[0,T]\right\}$ is the solution to the following SDE:
    	\begin{equation}\label{mdp2m}
    		\left\{\begin{aligned}
    			\mathrm{d}M^{\psi_\epsilon}_t=&\frac{1}{\lambda(\epsilon)}\left(b_\epsilon(\lambda(\epsilon)M^{\psi_{\epsilon}}_t+X_t^0,\mathcal{L}_{X_t^\epsilon})-b(X_t^0,\mathcal{L}_{X_t^0})\right)\mathrm{d}t  \\
    			&+\frac{\sqrt{\epsilon}}{\lambda(\epsilon)}\sigma_{\epsilon}(\lambda(\epsilon)M^{\psi_{\epsilon}}+X^0_t,\mathcal{L}_{X_t^0})\mathrm{d}W_t  \\
    			&+\sigma_{\epsilon}(\lambda(\epsilon)M^{\psi_{\epsilon}}+X^0_t,\mathcal{L}_{X_t^0})\psi_{\epsilon}(t)\mathrm{d}t-\mathrm{d}\hat{K}^{\epsilon,\psi_{\epsilon}}_t,  \\
    			\,\,\,\,M^{\psi_{\epsilon}}_0=0.
    		\end{aligned}\right.
    	\end{equation}

    	By Theorem 2.11, it is sufficient to verify the following two claims:
    	
    	\noindent\textbf{(MDP)}$\bm{_{1}}$ For any given $m\in(0,\infty)$, let $\left\{\psi_n, n\in\mathbb{N}\right\}, \psi\in\mathcal{S}_m$ be such that $\psi_n\to\psi$ in $\mathcal{S}_m$ as $n\to\infty$, Then
    	$$\lim\limits_{n\to\infty}\sup_{t\in[0,T]}|\Gamma^0(\psi_n)(t)-\Gamma^0(\psi)(t)|=0.$$
    	
    	\noindent\textbf{(MDP)}$\bm{_{2}}$ For any given $m\in(0,\infty)$, let $\left\{\psi_{\epsilon},\epsilon>0\right\}\in\mathcal{D}_m^T$, then for any $\xi>0$,
    	$$\lim\limits_{\epsilon\to0}P\big(\sup_{t\in[0,T]}|M^{\psi_{\epsilon}}(t)-\Gamma^0(\psi_{\epsilon})(t)|>\xi\big)=0.$$
    	
    	The verifications of \textbf{(MDP)}$\bm{_{1}}$ and \textbf{(MDP)}$\bm{_{2}}$ will be given in the following Propositions respectively.
    	
    \end{proof}

    \begin{proposition}[\textbf{MDP}$\bm{_{1}}$]
    	For any given $m\in(0,\infty)$, let $\left\{\psi_n, n\in\mathbb{N}\right\}, \psi\in\mathcal{S}_m$ be such that $\psi_n\to\psi$ in $\mathcal{S}_m$ as $n\to\infty$, then
    	$$\lim\limits_{n\to\infty}\sup_{t\in[0,T]}|\Gamma^0(\psi_n)(t)-\Gamma^0(\psi)(t)|=0.$$
    	
    	\begin{proof}
    		Notice that $\nu^{\psi}=\Gamma^0(\psi)$ is the solution to (\ref{mdp1}), and $\nu^{\psi_n}=\Gamma^0(\psi_n)$ is the the solution to (\ref{mdp1}) with $\psi$ replaced by $\psi_n$. It is equivalent to prove the following result:
    		$$\lim\limits_{n\to\infty}\sup_{t\in[0,T]}|\nu^{\psi_n}(t)-\nu^\psi(t)|=0.$$
    		
    		The proof is similar to that of Proposition \ref{ld1}. First we will show that $\left\{\nu^{\psi_n}\right\}_{n\geq 1}$ is pre-compact in $C([0,T],\overline{D(A)})$. Note that (\ref{nulim}) implies that $\left\{\nu^{\psi_n}\right\}_{n\geq 1}$ is uniformly bounded. Let
    		\begin{align}\label{nulimm}
    			\sup_{n\geq 1}\sup_{t\in[0,T]}|\nu^{\psi_n}(t)|=C_m<\infty.
    		\end{align}
    	
    	    For any $0\leq s<t\leq T$, combining (\ref{si}) and (\ref{nulimm}), we have
    	    \begin{align*}
    	    	&|\nu^{\psi_n}_t-\nu^{\psi_n}_s|  \\
    	    	\leq &C_m\int_{s}^{t}|b'(X^0_r,\mathcal{L}_{X^0_r})|\mathrm{d}r+\left(\int_{s}^{t}|\psi(r)|^2\mathrm{d}r\right)^{\frac{1}{2}}\left(\int_{s}^{t}|\sigma(X^0_r,\mathcal{L}_{X^0_r})|^2\mathrm{d}r\right)^{\frac{1}{2}}-|\hat{K}^{\psi_n}|_s^t  \\
    	    	\leq&C_m\int_{s}^{t}|b'(X^0_r,\mathcal{L}_{X^0_r})|\mathrm{d}r+(2m)^{\frac{1}{2}}\left(\int_{s}^{t}|\sigma(X^0_r,\mathcal{L}_{X^0_r})|^2\mathrm{d}r\right)^{\frac{1}{2}}-|\hat{K}^{\psi_n}|_s^t.
    	    \end{align*}
    		By (\ref{slim}), (\ref{b1}) and (\ref{kd}) we can deduce that $\left\{\nu^{\psi_n}\right\}_{n\geq 1}$ is equi-continuous in $C([0,T],\overline{D(A)})$. Thus, $\left\{\nu^{\psi_n}\right\}_{n\geq 1}$ is pre-compact in $C([0,T],\overline{D(A)})$.
    		
    		Let $\tilde{\nu}$ be any limit of some subsequence of $\left\{\nu^{\psi_n}\right\}_{n\geq 1}$ in $C([0,T],\overline{D(A)})$ and $\tilde{K}$ be any limit of some subsequence of $\bigl\{\hat{K}^{\psi_n}\bigr\}_{n\geq 1}$. Next we will prove that $\tilde{\nu}=\nu^\psi$. 
    
    Without loss of generality, we assume
    		\begin{align}
    			\label{tnu}\lim\limits_{n\to\infty}\sup_{t\in[0,T]}|\tilde{\nu}(t)-\nu^{\psi_n}(t)|=0.  \\
    			\label{tk}\lim\limits_{n\to\infty}\sup_{t\in[0,T]}|\tilde{K}(t)-\hat{K}^{\psi_n}(t)|=0.
    		\end{align}
    		Due to (\ref{b1}) and (\ref{tnu}),
    		\begin{align*}
    			&\int_{0}^{T}|b'(X^0_t,\mathcal{L}_{X_t^0})\nu^{\psi_n}(t)-b'(X^0_t,\mathcal{L}_{X^0_t})\tilde{\nu}(t)|\mathrm{d}t \\
    			\leq &\sup_{t\in[0,T]}|\tilde{\nu}(t)-\nu^{\psi_n}(t)|\int_{0}^{T}|b'(X^0_t,\mathcal{L}_{X^0_t})|\mathrm{d}t \to 0, \quad n\to\infty.
    		\end{align*}
    		Thus, for each $t\in[0,T]$
    		\begin{align}\label{b'}
    			\int_{0}^{T}|b'(X^0_t,\mathcal{L}_{X^0_t})\nu^{\psi_n}(t)|\mathrm{d}t\to\int_{0}^{T}|b'(X^0_t,\mathcal{L}_{X^0_t})\tilde{\nu}(t)|\mathrm{d}t, \quad n\to\infty.
    		\end{align}
    		
    		Since $\psi_n\in \mathcal{S}_m$ and  $\psi_n\to\psi$ as $n\to\infty$, we have
    		\begin{align*}
    			&\int_{0}^{T}|\sigma(X^0_t,\mathcal{L}_{X_t^0})\psi_n(t)-\sigma(X^0_t,\mathcal{L}_{X^0_t})\psi(t)|\mathrm{d}t \\
    			\leq &\sup_{t\in[0,T]}|\psi_n(t)-\psi(t)|\int_{0}^{T}|\sigma(X^0_t,\mathcal{L}_{X^0_t})|\mathrm{d}t \to 0, \quad\text{as}\,\, n\to\infty.
    		\end{align*}
    		By using the similar argument as in the (\ref{ldp1c}) we can get
    		\begin{align}\label{mdp1s}
    			\int_{0}^{t}\sigma(X^0_s,\mathcal{L}_{X^0_s})\psi_n(s)\mathrm{~d}s\to \int_{0}^{t}\sigma(X^0_s,\mathcal{L}_{X^0_s})\psi(s)\mathrm{~d}s,\quad \text{as}\,\,\,n\to\infty.
    		\end{align}
    		Recall that $(\nu^{\psi_n},\hat{K}^{\psi_n})$ is the solution of (\ref{mdp1}) with $\psi$ replaced by $\psi_n$ :
    		\begin{align*}
    			\mathrm{d}\nu^{\psi_n}(t)=b'(X^0_t,\mathcal{L}_{X^0_t})\nu^{\psi}(t)\mathrm{d}t+\sigma(X_t^0,\mathcal{L}_{X_t^0})\psi_n(t)\mathrm{d}t-\mathrm{d}\hat{K}^{\psi_n}_t, \quad t\in[0,T].
    		\end{align*}
    		Letting $n\to\infty$ and combining the above estimates, we obtain that $(\tilde{\nu},\tilde{K})$ is a solution to (\ref{mdp1}). Thanks to the uniqueness of the solutions of (\ref{mdp1}) we can deduce that $\tilde{\nu}=\nu^\psi$, which completes the proof.
    		
    	\end{proof}
    \end{proposition}

   To verify \textbf{(MDP)}$\bm{_{2}}$, we will need the following lemma. 

    \begin{lemma}\label{mdp2lemma}
    	Let $M^{\psi_\epsilon}$ be the solution to (\ref{mdp2m}), then there exists some $\kappa_0>0$ such that
    	\begin{align}\label{minfty}
    		\sup_{\epsilon\in(0,\kappa_0]}\mathbb{E}\left(\sup_{t\in[0,T]}|M^{\psi_{\epsilon}}(t)|^2\right)<\infty.
    	\end{align}
        \begin{proof}
        	By applying It$\hat{\text{o}}$ formula we have for any $t\in[0,T]$,
        	\begin{align}
        		&|M^{\psi_{\epsilon}}(t)|^2  \nonumber\\
        		=&\frac{2}{\lambda(\epsilon)}\int_{0}^{t}\left\langle M^{\psi_{\epsilon}}(s),b_{\epsilon}(\lambda(\epsilon)M^{\psi_{\epsilon}}_s+X^0_s,\mathcal{L}_{X^{\epsilon}})-b(X^0_s,\mathcal{L}_{X^0_s})\right\rangle\mathrm{d}s  \nonumber\\
        		&+\frac{2\sqrt{\epsilon}}{\lambda(\epsilon)}\int_{0}^{t}\left\langle M^{\psi_{\epsilon}}(s),\sigma_{\epsilon}(\lambda(\epsilon)M^{\psi_{\epsilon}}_s+X^0_s,\mathcal{L}_{X^\epsilon})\mathrm{~d}W_s\right\rangle \nonumber+\frac{\epsilon}{\lambda^{2}(\epsilon)}\int_{0}^{t}\lVert \sigma_{\epsilon}(\lambda(\epsilon)M^{\psi_{\epsilon}}_s+X^0_s,\mathcal{L}_{X^\epsilon})\rVert^2\mathrm{~d}s  \nonumber\\
        		&+2\int_{0}^{t}\left\langle \sigma_{\epsilon}(\lambda(\epsilon)M^{\psi_{\epsilon}}_s+X^0_s,\mathcal{L}_{X^\epsilon})\psi_\epsilon(s),M^{\psi_{\epsilon}}_s\right\rangle\mathrm{d}s  \nonumber-2\int_{0}^{t}\left\langle M^{\psi_{\epsilon}}(s),\mathrm{d}\hat{K}^{\epsilon,\psi_{\epsilon}}_s\right\rangle \nonumber\\
        		=:&I_{1}(t)+I_{2}(t)+I_{3}(t)+I_{4}(t)+I_{5}(t).
        	\end{align}
        
        
            By (\hyperref[h1]{\textbf{H1}}), (\hyperref[h2]{\textbf{H2}}) and (\hyperref[b2]{\textbf{B2}}), there exists $\epsilon_{1}>0$ such that
            \begin{align}\label{ep1}
            	\frac{\epsilon}{\lambda^2(\epsilon)}\vee \lambda(\epsilon) \vee \rho_{b,\epsilon} \vee \rho_{\sigma,\epsilon} \vee \frac{\rho_{b,\epsilon}}{\lambda(\epsilon)}\in (0,\frac{1}{2}], \quad \forall\epsilon\in(0,\epsilon_{1}].
            \end{align}
            Let $\epsilon_{2}=\epsilon_{0}\wedge\epsilon_{1}\wedge\frac{1}{2}$, where $\epsilon_{0}$ is the constant given in Lemma \ref{ep0}. Denote by $C$  a generic constant which may change from line to line and is independent of $\epsilon$.

            Recall the definition of $M^{\psi_\epsilon}$, since $(X^0,K^0)$,  $(X^{\psi_\epsilon},K^{\epsilon,\psi_\epsilon})\in\mathcal{A}$, with (\ref{mono}) we can arrive that
            \begin{align}\label{mi5}
            	I_5(t)=\frac{2}{\lambda^2(\epsilon)}\int_{0}^{t}\left\langle X^{\psi_\epsilon}_s-X^0_s,\mathrm{d}K^0_s-\mathrm{d}K^{\epsilon,\psi_{\epsilon}}_s\right\rangle \leq\,\,0.
            \end{align}

            Due to (\hyperref[h1]{\textbf{H1}}), (\hyperref[h1]{\textbf{H1}}), Lemma \ref{ep0} and (\ref{slim}), for any $\epsilon\in(0,\epsilon_{2}]$,
            \begin{align}\label{mi1}
            	I_1(t)=&\,\,\frac{2}{\lambda(\epsilon)}\int_{0}^{t}\left\langle b_{\epsilon}(\lambda(\epsilon)M^{\psi_{\epsilon}}_s+X^0_s,\mathcal{L}_{X^{\epsilon}})-b(X^0_s,\mathcal{L}_{X^0_s}),M^{\psi_{\epsilon}}(s)\right\rangle\mathrm{d}s  \nonumber\\
            	=&\,\,\frac{2}{\lambda(\epsilon)}\int_{0}^{t}\left\langle b_{\epsilon}(\lambda(\epsilon)M^{\psi_{\epsilon}}_s+X^0_s,\mathcal{L}_{X^{\epsilon}})-b(\lambda(\epsilon)M^{\psi_{\epsilon}}_s+X^0_s,\mathcal{L}_{X^{\epsilon}}),M^{\psi_{\epsilon}}(s)\right\rangle\mathrm{d}s  \nonumber\\
            	&+\frac{2}{\lambda(\epsilon)}\int_{0}^{t}\left\langle b(\lambda(\epsilon)M^{\psi_{\epsilon}}_s+X^0_s,\mathcal{L}_{X^{\epsilon}})-b(X^0_s,\mathcal{L}_{X^{\epsilon}}),M^{\psi_{\epsilon}}(s)\right\rangle\mathrm{d}s  \nonumber\\
            	&+\frac{2}{\lambda(\epsilon)}\int_{0}^{t}\left\langle b(X^0_s,\mathcal{L}_{X^{\epsilon}})-b(X^0_s,\mathcal{L}_{X^0_s}),M^{\psi_{\epsilon}}(s)\right\rangle\mathrm{d}s \nonumber\\
            	\leq&\,\,\frac{2\rho_{b,\epsilon}}{\lambda(\epsilon)}|M^{\psi_{\epsilon}}_s|\mathrm{d}s+2L\int_{0}^{t}(1+|X^0_s|^q+|\lambda(\epsilon)M^{\psi_{\epsilon}}_s+X^0_s|^q)|M^{\psi_{\epsilon}}_s|^2\mathrm{d}s  \nonumber\\
            	&+\frac{2L}{\lambda(\epsilon)}\int_{0}^{t}|M^{\psi_{\epsilon}}_s|\mathbb{W}(\mathcal{L}_{X^{\epsilon}_s},\mathcal{L}_{X^0_s})\mathrm{d}s  \nonumber\\
            	\leq&\,\,C\int_{0}^{t}|M^{\psi_{\epsilon}}_s|^2\mathrm{d}s+(\frac{2\rho_{b,\epsilon}}{\lambda(\epsilon)}+\frac{2LC_T\sqrt{\epsilon}}{\lambda(\epsilon)})\int_{0}^{t}|M^{\psi_{\epsilon}}_s|\mathrm{d}s  \nonumber\\
            	\leq&\,\, C\int_{0}^{t}|M^{\psi_{\epsilon}}_s|^2\mathrm{d}s+C.
            \end{align}

            For $I_3(t)$, we have
            \begin{align}\label{mi3}
            	I_3(t)=&\,\,\frac{\epsilon}{\lambda^{2}(\epsilon)}\int_{0}^{t}\lVert \sigma_{\epsilon}(\lambda(\epsilon)M^{\psi_{\epsilon}}_s+X^0_s,\mathcal{L}_{X_s^\epsilon})\rVert^2\mathrm{~d}s  \nonumber\\
            	\leq&\,\, \frac{\epsilon}{\lambda^{2}(\epsilon)}\int_{0}^{t}\lVert \sigma_{\epsilon}(\lambda(\epsilon)M^{\psi_{\epsilon}}_s+X^0_s,\mathcal{L}_{X_s^\epsilon})-\sigma(\lambda(\epsilon)M^{\psi_{\epsilon}}_s+X^0_s,\mathcal{L}_{X_s^\epsilon})\rVert^2\mathrm{~d}s  \nonumber\\
            	&+\frac{\epsilon}{\lambda^{2}(\epsilon)}\int_{0}^{t}\lVert \sigma(\lambda(\epsilon)M^{\psi_{\epsilon}}_s+X^0_s,\mathcal{L}_{X_s^\epsilon})-\sigma(X^0_s,\mathcal{L}_{X^0_s})\rVert^2\mathrm{~d}s  \nonumber\\
            	&+\frac{\epsilon}{\lambda^{2}(\epsilon)}\int_{0}^{t}\lVert \sigma(X^0_s,\mathcal{L}_{X^0_s})\rVert^2\mathrm{~d}s  \nonumber \\
            	\leq&\,\, \frac{C\epsilon\rho_{\sigma,\epsilon}}{\lambda^{2}(\epsilon)}+
            	C\epsilon\int_{0}^{t}|M^{\psi_{\epsilon}}_s|^2\mathrm{d}s+\frac{\epsilon}{\lambda^{2}(\epsilon)}\int_{0}^{t}\mathbb{W}^2(\mathcal{L}_{X^{\epsilon}_s},\mathcal{L}_{X^0_s})\mathrm{d}s  \nonumber\\
            	&+\frac{\epsilon}{\lambda^{2}(\epsilon)}\int_{0}^{t}\lVert \sigma(X^0_s,\mathcal{L}_{X^0_s})\rVert^2\mathrm{~d}s  \nonumber \\
            	\leq&\,\, C\int_{0}^{t}|M^{\psi_{\epsilon}}_s|^2\mathrm{d}s+C.
            \end{align}

            For $I_4(t)$,  we have 
            \begin{align}\label{mi44}
            	I_4(t)=&\,2\int_{0}^{t}\left\langle \sigma_{\epsilon}(\lambda(\epsilon)M^{\psi_{\epsilon}}_s+X^0_s,\mathcal{L}_{X^\epsilon})\psi_\epsilon(s),M^{\psi_{\epsilon}}_s\right\rangle\mathrm{d}s  \nonumber\\
            	=&\,2\int_{0}^{t}\left[\sigma_{\epsilon}(\lambda(\epsilon)M^{\psi_{\epsilon}}_s+X^0_s,\mathcal{L}_{X^\epsilon})-\sigma(\lambda(\epsilon)M^{\psi_{\epsilon}}_s+X^0_s,\mathcal{L}_{X^\epsilon_s})\right]\psi_\epsilon(s),M^{\psi_{\epsilon}}_s\mathrm{d}s  \nonumber\\
            	&+2\int_{0}^{t}\left\langle \left[\sigma(\lambda(\epsilon)M^{\psi_{\epsilon}}_s+X^0_s,\mathcal{L}_{X^\epsilon})-\sigma(X^0_s,\mathcal{L}_{X^0_s})\right]\psi_\epsilon(s),M^{\psi_{\epsilon}}_s\right\rangle\mathrm{d}s  \nonumber \\
            	&+2\int_{0}^{t}\left\langle \sigma(X^0_s,\mathcal{L}_{X^0_s})\psi_\epsilon(s),M^{\psi_{\epsilon}}_s\right\rangle\mathrm{d}s  \nonumber\\
            \leq&\,\,2\rho_{\sigma,\epsilon}\int_{0}^{t}|M^{\psi_{\epsilon}}_s||\psi_\epsilon(s)|\mathrm{d}s+C\int_{0}^{t}(\lambda(\epsilon)|M^{\psi_{\epsilon}}_s|+\mathbb{W}_2(\mathcal{L}_{X^\epsilon_s},\mathcal{L}_{X^0_s}))|M^{\psi_{\epsilon}}_s||\psi_\epsilon(s)|\mathrm{d}s  \nonumber\\
            	&+\,C\int_{0}^{t}\lVert \sigma(X^0_s,\mathcal{L}_{X^0_s})\rVert|M^{\psi_{\epsilon}}_s||\psi_\epsilon(s)|\mathrm{d}s . \nonumber
            \end{align}
            By H{\"o}lder inequality, Young's inequality and $\psi_\epsilon\in\mathcal{S}_m$, we get
            \begin{align}\label{mi4}	I_4(t)
            	\leq&\,\, \left(1+C\left(\mathbb{E}(\sup_{s\in[0,T]}|X^\epsilon_s-X^0_s|^2
            	)\right)^{\frac{1}{2}}\right)\int_{0}^{t}|M^{\psi_{\epsilon}}_s||\psi_\epsilon(s)|\mathrm{d}s+C\lambda(\epsilon)\int_{0}^{t}|M^{\psi_{\epsilon}}_s|^2|\psi_\epsilon(s)|\mathrm{d}s   \nonumber\\
            	&+\, C\int_{0}^{t}\lVert \sigma(X^0_s,\mathcal{L}_{X^0_s})\rVert|M^{\psi_{\epsilon}}_s||\psi_\epsilon(s)|\mathrm{d}s  \nonumber\\
            	\leq&\,\, C\int_{0}^{t}|M^{\psi_{\epsilon}}_s|^2\mathrm{d}s+C\int_{0}^{t}|\psi_\epsilon(s)|^2\mathrm{d}s+C\int_{0}^{t}|M^{\psi_{\epsilon}}_s|^2|\psi_\epsilon(s)|^2\mathrm{d}s+C\int_{0}^{t}\lVert \sigma(X^0_s,\mathcal{L}_{X^0_s})\rVert^2\mathrm{d}s \nonumber\\
            	\leq&\,\, C\int_{0}^{t}\left(1+|\psi_\epsilon(s)|^2\right)|M^{\psi_{\epsilon}}_s|^2\mathrm{d}s+C.
            \end{align}

            Combining the above estimates (\ref{mi5}), (\ref{mi1}), (\ref{mi3}) and (\ref{mi4}) together and by applying Gronwall's inequality we obtain that for any $\epsilon\in(0,\epsilon_{2}]$, $t\in[0,T]$,
            \begin{align}
            	|M^{\psi_{\epsilon}}(t)|^2\leq e^{\int_{0}^{T}(1+|\psi_\epsilon(s)|^2)\mathrm{d}s}\left\{C+\sup_{s\in[0,T]}|I_2(t)|\right\}.
            \end{align}

            Since $\psi_\epsilon\in\mathcal{S}_m$ $P$-a.s., then for any $\epsilon\in(0,\epsilon_{2}]$ we have
            \begin{align}
            	\frac{1}{2}\int_{0}^{T}|\psi_\epsilon(s)|^2\mathrm{d}s\leq m, \quad P\text{-a.s.}.
            \end{align}

            Therefore, there exists a constant $\zeta\in(0,\infty)$ such that for any $\epsilon\in(0,\epsilon_{2}]$,
            \begin{align}\label{mgron}
            	\mathbb{E}\left(\sup_{t\in[0,T]}|M^{\psi_\epsilon}(t)|^2\right)\leq \zeta\left\{1+\mathbb{E}\left(\sup_{t\in[0,T]}|I_2(t)|\right)\right\}.
            \end{align}

            For $I_2(t)$, By Burkholder-Davis-Gundy’s inequality, (\hyperref[h1]{\textbf{H1}}), (\hyperref[h2]{\textbf{H2}}), Young's inequality, Lemma \ref{ep0}, (\ref{slim}) and (\ref{mi3}), for any $\epsilon\in(0,\epsilon_{2}]$,
            \begin{align}\label{mi2}
            	\mathbb{E}\left(\sup_{t\in[0,T]}|I_2(t)|\right) \nonumber
            	\leq&\,\, \frac{C\sqrt{\epsilon}}{\lambda(\epsilon)}\mathbb{E}\left(\int_{0}^{T}|M^{\psi_\epsilon}_s|^2\lVert\sigma_{\epsilon}(\lambda(\epsilon)M^{\psi_\epsilon}_s+X^0_s,\mathcal{L}_{X^\epsilon_s})\rVert^2\mathrm{d}s\right)^{\frac{1}{2}} \nonumber\\
            	\leq&\,\, \frac{C\sqrt{\epsilon}}{\lambda(\epsilon)}\mathbb{E}\left(\sup_{s\in[0,T]}|M^{\psi_\epsilon}_s|^2\right)+\frac{C\sqrt{\epsilon}}{\lambda(\epsilon)}\mathbb{E}\int_{0}^{T}\lVert\sigma_{\epsilon}(\lambda(\epsilon)M^{\psi_\epsilon}_s+X^0_s,\mathcal{L}_{X^\epsilon_s})\rVert^2\mathrm{d}s  \nonumber\\
            	\leq&\,\, \frac{C\sqrt{\epsilon}}{\lambda(\epsilon)}\mathbb{E}\left(\sup_{s\in[0,T]}|M^{\psi_\epsilon}_s|^2\right)+\frac{C\sqrt{\epsilon}}{\lambda(\epsilon)}\left(C\rho_{\sigma,\epsilon}^2+C\lambda^2(\epsilon)\mathbb{E}\int_{0}^{T}|M^{\psi_\epsilon}_s|^2\mathrm{d}s\right)  \nonumber\\
            	&+\, \frac{C\sqrt{\epsilon}}{\lambda(\epsilon)}\left(\int_{0}^{T}\mathbb{W}_2^2(\mathcal{L}_{X_s^\epsilon},\mathcal{L}_{X_s^0})\mathrm{d}s+\int_{0}^{T}\lVert\sigma(X^0_s,\mathcal{L}_{X^0_s})\rVert^2\mathrm{d}s\right)  \nonumber \\
            	\leq&\,\, C\left(\frac{\sqrt{\epsilon}}{\lambda(\epsilon)}+\sqrt{\epsilon}\lambda(\epsilon)\right)\mathbb{E}\left(\sup_{s\in[0,T]}|M^{\psi_\epsilon}_s|^2\right)+C.
            \end{align}

            By substituting (\ref{mi2}) back into (\ref{mgron}), we obtain for any $\epsilon\in(0,\epsilon_{2}]$,
            \begin{align}
            	\left(1-\frac{C\sqrt{\epsilon}}{\lambda(\epsilon)}-C\sqrt{\epsilon}\lambda(\epsilon)\right)\mathbb{E}\left(\sup_{s\in[0,T]}|M^{\psi_\epsilon}_s|^2\right)\leq C.
            \end{align}

            Then there exists a constant $\kappa_0>0$ such that
            \begin{align*}
            	\sup_{\epsilon\in(0,\kappa_0]}\mathbb{E}\left(\sup_{t\in[0,T]}|M^{\psi_{\epsilon}}(t)|^2\right)<\infty,
            \end{align*}
            which completes the proof.
        \end{proof}

    \end{lemma}

    Finally we are in the position to verify \textbf{(MDP)}$\bm{_{2}}$.

    \begin{proposition}[\textbf{MDP}$\bm{_{2}}$]
    	For any given $m\in(0,\infty)$, $\psi_\epsilon\in\mathcal{D}^T_m$, for any $\xi>0$, we have
    	\begin{align*}
    		\lim\limits_{\epsilon\to 0}P\left(\sup_{t\in[0,T]}|M^{\psi_\epsilon}_t-\nu^{\psi_\epsilon}_t|\geq\xi\right)=0.
    	\end{align*}

        \begin{proof}
        	For each fixed $\epsilon>0$ and $a\in\mathbb{N}$, we can define a stopping time
        	\begin{align}
        		\tau_\epsilon^a=\inf\big\{t\geq 0:|M^{\psi_\epsilon}(t)|\geq a\big\}\wedge T.
        	\end{align}
            Due to Lemma \ref{mdp2lemma} and Markov's inequality, we get
            \begin{align}\label{cheby}
            	P(\tau_\epsilon^a<T)\leq\frac{\mathbb{E}(\sup_{t\in[0,T]}|M^{\psi_\epsilon}(t)|^2)}{a^2}\leq \frac{C}{a^2}, \quad \forall\epsilon\in(0,\kappa_0].
            \end{align}

            Denote $Q^\epsilon(t)=M^{\psi_\epsilon}(t)-\nu^{\psi_\epsilon}(t)$ for each $t\in[0,T]$. We have
            \begin{align}
            	\mathrm{d}Q^\epsilon_t=&\,\,\mathrm{d}M^{\psi_\epsilon}_t-\mathrm{d}\nu^{\psi_\epsilon}_t   \nonumber\\
            	=&\left(\frac{1}{\lambda(\epsilon)}\left(b_\epsilon(\lambda(\epsilon)M^{\psi_\epsilon}_t+X^0_t,\mathcal{L}_{X^\epsilon_t})-b(X^0_t,\mathcal{L}_{X^0_t})\right)-b'(X^0_t,\mathcal{L}_{X^0_t})\nu^{\psi_\epsilon}_t\right)\mathrm{d}t   \nonumber\\
            	&+\left(\sigma_{\epsilon}(\lambda(\epsilon)M^{\psi_\epsilon}_t+X^0_t,\mathcal{L}_{X^\epsilon_t})-\sigma(X^0_t,\mathcal{L}_{X^0_t})\right)\psi_\epsilon(t)\mathrm{d}t  \nonumber \\
            	&+\frac{\sqrt{\epsilon}}{\lambda(\epsilon)}\sigma_{\epsilon}(\lambda(\epsilon)M^{\psi_\epsilon}_t+X^0_t,\mathcal{L}_{X^\epsilon_t})\mathrm{d}W_t  \nonumber-\left(\mathrm{d}\hat{K}^{\epsilon,\psi_\epsilon}-\mathrm{d}\hat{K}^{\psi_\epsilon}\right)
            \end{align}
            By It$\hat{\text{o}}$'s formula, we get
            \begin{align}\label{mdpq}
            	|Q^\epsilon_{t\wedge\tau_\epsilon^a}|^2=&2\int_{0}^{t\wedge\tau_\epsilon^a}\left\langle \frac{1}{\lambda(\epsilon)}\left(b_\epsilon(\lambda(\epsilon)M^{\psi_\epsilon}_s+X^0_s,\mathcal{L}_{X^\epsilon_s})-b(X^0_s,\mathcal{L}_{X^0_s})\right)-b'(X^0_s,\mathcal{L}_{X^0_s})\nu^{\psi_\epsilon}_s,Q^\epsilon_s\right\rangle\mathrm{d}s  \nonumber\\
            	&+\frac{2\sqrt{\epsilon}}{\lambda(\epsilon)}\int_{0}^{t\wedge\tau_\epsilon^a}\left\langle Q^\epsilon_s,\sigma_{\epsilon}(\lambda(\epsilon)M^{\psi_\epsilon}_s+X^0_s,\mathcal{L}_{X^\epsilon_s})\mathrm{d}W_s\right\rangle  \nonumber\\
            	&+\frac{\epsilon}{\lambda^2(\epsilon)}\int_{0}^{t\wedge\tau_\epsilon^a}\lVert\sigma_{\epsilon}(\lambda(\epsilon)M^{\psi_\epsilon}_s+X^0_s,\mathcal{L}_{X^\epsilon_s})\rVert^2\mathrm{d}s  \nonumber\\
            	&+2\int_{0}^{t\wedge\tau_\epsilon^a}\left\langle \left(\sigma_{\epsilon}(\lambda(\epsilon)M^{\psi_\epsilon}_s+X^0_s,\mathcal{L}_{X^\epsilon_s})-\sigma(X^0_s,\mathcal{L}_{X^0_s})\right)\psi_\epsilon(s),Q^\epsilon_s\right\rangle\mathrm{d}s  \nonumber\\
            	&-2\int_{0}^{t\wedge\tau_\epsilon^a}\left\langle Q^\epsilon_s,\mathrm{d}\hat{K}^{\epsilon,\psi_\epsilon}-\mathrm{d}\hat{K}^{\psi_\epsilon}\right\rangle  \nonumber\\
            	=:&\,\,\hat{I}_1(t)+\hat{I}_2(t)+\hat{I}_3(t)+\hat{I}_4(t)+\hat{I}_5(t)
            \end{align}

            For $\hat{I}_{5}(t)$, by (\ref{mono}) we have for each $\epsilon\in(0,\epsilon_{3}]$,
            \begin{align}\label{hati5}
            	\hat{I}_{5}(t)=-2\int_{0}^{t\wedge\tau_\epsilon^a}\left\langle Q^\epsilon_s,\mathrm{d}\hat{K}^{\epsilon,\psi_\epsilon}-\mathrm{d}\hat{K}^{\psi_\epsilon}\right\rangle\leq 0.
            \end{align}

            By (\ref{nulim}) and $\psi_\epsilon\in\mathcal{D}_m^T$, there exists some $\Omega^0\in\mathcal{F}$ with $P(\Omega^0)=1$ such that
            \begin{align}\label{gamma}
            	\gamma:=\sup_{\epsilon\in(0,\kappa_0]}\sup_{\omega\in\Omega^0, t\in[0,T]}|\nu^{\psi_\epsilon}_t(\omega)|<\infty.
            \end{align}

            For $\hat{I}_1(t)$,
            \begin{align}\label{hati1}
            	\hat{I}_1(t)=&\frac{2}{\lambda(\epsilon)}\int_{0}^{t\wedge\tau_\epsilon^a}\left\langle b_\epsilon(\lambda(\epsilon)M^{\psi_\epsilon}_s+X^0_s,\mathcal{L}_{X^\epsilon_s})-b(\lambda(\epsilon)M^{\psi_\epsilon}_s+X^0_s,\mathcal{L}_{X^\epsilon_s}),Q^{\psi_\epsilon}_s\right\rangle\mathrm{d}s  \nonumber\\
            	&+\frac{2}{\lambda(\epsilon)}\int_{0}^{t\wedge\tau_\epsilon^a}\left\langle b(\lambda(\epsilon)M^{\psi_\epsilon}_s+X^0_s,\mathcal{L}_{X^\epsilon_s})-b(\lambda(\epsilon)M^{\psi_\epsilon}_s+X^0_s,\mathcal{L}_{X^0_s}),Q^{\psi_\epsilon}_s\right\rangle\mathrm{d}s  \nonumber\\
            	&+2\int_{0}^{t\wedge\tau_\epsilon^a}\left\langle \frac{1}{\lambda(\epsilon)}\left(b(\lambda(\epsilon)M^{\psi_\epsilon}_s+X^0_s,\mathcal{L}_{X^0_s})-b(X^0_s,\mathcal{L}_{X^0_s})\right)-b'(X^0_s,\mathcal{L}_{X^0_s})\nu^{\psi_\epsilon}_s,Q^{\psi_\epsilon}_s\right\rangle\mathrm{d}s  \nonumber\\
            	=&\,\,\hat{I}_{1,1}(t)+\hat{I}_{1,2}(t)+\hat{I}_{1,3}(t).
            \end{align}
            Due to (\hyperref[h1]{\textbf{H1}}), (\hyperref[h2]{\textbf{H2}}) and Lemma \ref{ep0} we have
            \begin{align}\label{hati112}
            	\hat{I}_{1,1}(t)+\hat{I}_{1,2}(t)\leq&\, \frac{2\rho_{b,\epsilon}}{\lambda(\epsilon)}\int_{0}^{t\wedge\tau_\epsilon^a}|Q^{\psi_\epsilon}_s|\mathrm{d}s+\frac{2L}{\lambda(\epsilon)}\int_{0}^{t\wedge\tau_\epsilon^a}\left(\mathbb{E}|X^\epsilon_s-X^0_s|^2\right)^{\frac{1}{2}}|Q^{\psi_\epsilon}_s|\mathrm{d}s  \nonumber\\
            	\leq& \, \frac{2\rho_{b,\epsilon}+2L(\epsilon+\rho_{b,\epsilon}^2+\epsilon\rho_{\sigma,\epsilon}^2)^{\frac{1}{2}}}{\lambda(\epsilon)}(a+\gamma)T.
            \end{align}
            Let $\epsilon_{3}=\kappa_0\wedge\epsilon_{2}$, then for any $\epsilon\in(0,\epsilon_{3}]$, by using the mean value theorem, (\hyperref[b0]{\textbf{B0}}) and (\hyperref[b1]{\textbf{B1}}), we obtain that there exists $\theta_{\epsilon}(s)\in[0,1]$ such that
            \begin{align*}
            	\hat{I}_{1,3}(t)=&\,\,2\int_{0}^{t\wedge\tau_{\epsilon}^{a}}\left\langle \frac{b(\lambda(\epsilon)M^{\psi_\epsilon}_s+X^{0}_s,\mathcal{L}_{X^{0}_s})-b(X^{0}_s,\mathcal{L}_{X^{0}_s})}{\lambda(\epsilon)M^{\psi_\epsilon}_s}M^{\psi_\epsilon}_s-b'(X^{0}_s,\mathcal{L}_{X^{0}_s})\nu^{\psi_\epsilon}_s,Q^{\psi_\epsilon}_s\right\rangle\mathrm{d}s \nonumber\\
            	\leq&\,\,2\int_{0}^{t\wedge\tau_{\epsilon}^{a}}\left\langle b'(\lambda(\epsilon)M^{\psi_\epsilon}_s\theta_{\epsilon}(s)+X^{0}_s,\mathcal{L}_{X^0_s})M^{\psi_\epsilon}_s-b'(X^0_s,\mathcal{L}_{X^0_s})\nu^{\psi_\epsilon}_s,Q^{\psi_\epsilon}_s\right\rangle\mathrm{d}s  \nonumber\\
            	=&\, 2\int_{0}^{t\wedge\tau_{\epsilon}^{a}}\left\langle b'(\lambda(\epsilon)M^{\psi_\epsilon}_s\theta_{\epsilon}(s)+X^0_s,\mathcal{L}_{X^0_s})M^{\psi_\epsilon}_s-b'(X^0_s,\mathcal{L}_{X^0_s})M^{\psi_\epsilon}_s,Q^{\psi_\epsilon}_s\right\rangle\mathrm{d}s  \nonumber\\
            	&+\, 2\int_{0}^{t\wedge\tau_{\epsilon}^{a}}\left\langle b'(X^0_s,\mathcal{L}_{X^0_s})M^{\psi_\epsilon}_s-b'(X^0_s,\mathcal{L}_{X^0_s})\nu^{\psi_\epsilon}_s,Q^{\psi_\epsilon}_s\right\rangle\mathrm{d}s  \nonumber\\
            	\leq&\,\, 2\int_{0}^{t\wedge\tau_{\epsilon}^{a}}| b'(\lambda(\epsilon)M^{\psi_\epsilon}_s\theta_{\epsilon}(s)+X^0_s,\mathcal{L}_{X^0_s})-b'(X^0_s,\mathcal{L}_{X^0_s})||M^{\psi_\epsilon}_s||Q^{\psi_\epsilon}_s|\mathrm{d}s   \nonumber\\
            	&+\, 2\int_{0}^{t\wedge\tau_{\epsilon}^{a}}|b'(X^0_s,\mathcal{L}_{X^0_s})||Q^{\psi_\epsilon}_s|^{2}\mathrm{d}s \nonumber\\
            	\leq&\,\, L'\lambda(\epsilon)\int_{0}^{t\wedge\tau_{\epsilon}^{a}}\left(1+|X^0_s|^{q'}+|\lambda(\epsilon)M^{\psi_\epsilon}_s+X^0_s|^{q'}\right)|M^{\psi_\epsilon}_s|^{2}|Q^{\psi_\epsilon}_s|\mathrm{d}s  \nonumber\\
            	&+\, 2\int_{0}^{t\wedge\tau_{\epsilon}^{a}}|b'(X^0_s,\mathcal{L}_{X^0_s})||Q^{\psi_\epsilon}_s|^{2}\mathrm{d}s
            \end{align*}
             Denote $C_a=L'\left(1+\sup_{s\in[0,T]}|X_{s}^{0}|^{q'}+|a+\sup_{s\in[0,T]}|X_{s}^{0}||^{q'}\right)a^{2}(a+\gamma)T$.  Notice that $C_a$ is independent of $\epsilon$, thus we get
            \begin{align}\label{hati132}
            	\hat{I}_{1,3}(t)\leq\, C_a\lambda(\epsilon)+2\int_{0}^{t}|b'(X^0_{s\wedge\tau_\epsilon^a},\mathcal{L}_{X^0_{s\wedge\tau_\epsilon^a}})||Q^{\psi_\epsilon}_{s\wedge\tau_\epsilon^a}|^2\mathrm{d}s
            \end{align}

            Substituting (\ref{hati112}) and (\ref{hati132}) back into (\ref{mdpq}), by Gronwall's inequality we have for any $\epsilon\in(0,\epsilon_{3}],\,\, t\in[0,T]$,

            \begin{align}\label{hatgron}
            	\sup_{t\in[0,T]}|Q^\epsilon_{t\wedge\tau_\epsilon^a}|^2\leq e^{2\int_{0}^{T}|b'(X^0_s,\mathcal{L}_{X^0_s})|\mathrm{d}s}\left[C_a(\frac{\rho_{b,\epsilon}+\lambda^2(\epsilon)+(\epsilon+\rho_{b,\epsilon}^2+\epsilon\rho_{\sigma,\epsilon})^{\frac{1}{2}}}{\lambda(\epsilon)})+\sum_{i=2}^{4}\hat{I}_i(t)\right]
            \end{align}
            Due to (\hyperref[b1]{\textbf{B1}}) we have
            \begin{align}
            	\Xi:=e^{2\int_{0}^{T}|b'(X^0_s,\mathcal{L}_{X^0_s})|\mathrm{d}s}<\infty.
            \end{align}

            For $\hat{I}_{2}(t)$, by Burkholder-Davis-Gundy’s inequality and (\ref{slim}), we obtain that
            \begin{align}\label{hati23}
            	&\mathbb{E}\left(\sup_{t\in[0,T]}|\hat{I}_{2}(t)|\right)+\mathbb{E}\left(\sup_{t\in[0,T]}|\hat{I}_{3}(t)|\right)  \nonumber\\
            	\leq&\,\frac{2\sqrt{\epsilon}}{\lambda(\epsilon)}\mathbb{E}\left(\int_{0}^{t\wedge\tau_\epsilon^a} |Q^\epsilon_s|^2\lVert\sigma_{\epsilon}(\lambda(\epsilon)M^{\psi_\epsilon}_s+X^0_s,\mathcal{L}_{X^\epsilon_s})\rVert^2\mathrm{d}s\right)^\frac{1}{2}  \nonumber\\
            	\leq&\,\, \frac{1}{4}\mathbb{E}\left(\sup_{s\in[0,T]}|Q^\epsilon_{s\wedge\tau_\epsilon^a}|^2\right)+\frac{C\epsilon}{\lambda^2(\epsilon)}\mathbb{E}\int_{0}^{t\wedge\tau_\epsilon^a}\lVert\sigma_{\epsilon}(\lambda(\epsilon)M^{\psi_\epsilon}_s+X^0_s,\mathcal{L}_{X^\epsilon_s})\rVert^2\mathrm{d}s \nonumber\\
            	\leq&\,\, \frac{1}{4}\mathbb{E}\left(\sup_{s\in[0,T]}|Q^\epsilon_{s\wedge\tau_\epsilon^a}|^2\right)+\frac{C\epsilon\rho_{\sigma,\epsilon}^2}{\lambda^2(\epsilon)}+\frac{C\epsilon}{\lambda^2(\epsilon)}\mathbb{E}\int_{0}^{t\wedge\tau_\epsilon^a}|M^{\psi_\epsilon}_s|^2\mathrm{d}s  \nonumber\\
            	&+\,\, \frac{C\epsilon}{\lambda^2(\epsilon)}\mathbb{E}\int_{0}^{t\wedge\tau_\epsilon^a}\mathbb{E}(|X^\epsilon_s-X^0_s|^2)\mathrm{d}s+\frac{C\epsilon}{\lambda^2(\epsilon)}\int_{0}^{T}\lVert\sigma(X^0_s,\mathcal{L}_{X^0_s})\rVert^2\mathrm{d}s  \nonumber\\
            	\leq&\,\, \frac{1}{4}\mathbb{E}\left(\sup_{s\in[0,T]}|Q^\epsilon_{s\wedge\tau_\epsilon^a}|^2\right)+\frac{C_a\epsilon}{\lambda^2(\epsilon)}(\rho_{\sigma,\epsilon}^2+\epsilon+\epsilon\rho_{\sigma,\epsilon}^2+\rho_{b,\epsilon}^2+T)
            \end{align}

            For $\hat{I}_{4}(t)$, under \hyperref[h1]{\textbf{(H1)}}, (\ref{gamma}), Lemma \ref{ep0} and $\psi_\epsilon\in\mathcal{D}_m^T$, and by using the similar arguments in the proof of (\ref{mi4}) we have for any $\epsilon\in(0,\epsilon_{3}]$,
            \begin{align}\label{hati4}
            	\mathbb{E}\left(\sup_{t\in[0,T]}|\hat{I}_{4}(t)|\right) \leq&\, C\rho_{\sigma,\epsilon}\mathbb{E}\int_{0}^{t\wedge\tau_\epsilon^a}|Q^\epsilon_{s}||\psi_\epsilon(s)|\mathrm{d}s  \nonumber\\
            	&+\,C\mathbb{E}\int_{0}^{t\wedge\tau_\epsilon^a}\left(\lambda(\epsilon)|M^{\psi_\epsilon}_{s}|+\mathbb{W}(\mathcal{L}_{X^\epsilon_s},\mathcal{L}_{X^0_s})+\lVert\sigma(X^0_s,\mathcal{L}_{X^0_s})\rVert\right)|\psi_\epsilon(s)||Q^\epsilon_s|\mathrm{d}s  \nonumber\\
            	\leq&\,\, CT(a+\gamma)^2\left(\rho_{\sigma,\epsilon}+\lambda(\epsilon)+(\epsilon+\rho_{b,\epsilon}^2+\epsilon\rho_{\sigma,\epsilon}^2)^{\frac{1}{2}}\right)\mathbb{E}\left(\int_{0}^{T}|\psi_\epsilon|^2\mathrm{d}s\right)^{\frac{1}{2}}   \nonumber\\
            	\leq&\,\, C_a\left(\rho_{\sigma,\epsilon}+\lambda(\epsilon)+(\epsilon+\rho_{b,\epsilon}^2+\epsilon\rho_{\sigma,\epsilon}^2)^{\frac{1}{2}}\right)
            \end{align}

            Inserting the above inequalities (\ref{hati23}) and (\ref{hati4}) into (\ref{hatgron}) we deduce that for any $\epsilon\in(0,\epsilon_{3}]$,
            \begin{align}
            	\frac{3}{4}\mathbb{E}\left(\sup_{t\in[0,T]}|Q^\epsilon_{t\wedge\tau_\epsilon^a}|^2\right)\leq C_a\left(\rho_{\sigma,\epsilon}+\lambda(\epsilon)+\frac{\rho_{b,\epsilon}}{\lambda(\epsilon)}+\frac{\epsilon}{\lambda^2(\epsilon)}+\frac{1+\lambda(\epsilon)}{\lambda(\epsilon)}(\epsilon+\rho_{b,\epsilon}^2+\epsilon\rho_{\sigma,\epsilon}^2)^{\frac{1}{2}}\right)
            \end{align}
            Due to (\hyperref[b2]{\textbf{B2}}), we  obtain
            \begin{align}
            	\lim\limits_{\epsilon\to 0}\mathbb{E}\left(\sup_{t\in[0,T]}|M^{\psi_\epsilon}_{t\wedge\tau_\epsilon^a}-\nu^{\psi_\epsilon}_{t\wedge\tau_\epsilon^a}|^2\right)=\lim\limits_{\epsilon\to 0}\mathbb{E}\left(\sup_{t\in[0,T]}|Q^{\epsilon}_{t\wedge\tau_\epsilon^a}|^2\right)=0.
            \end{align}

            For any $\xi>0$, $\epsilon\in(0,\epsilon_{3}]$, $a\in\mathbb{N}$, we have
            \begin{align}
            	&P\left(\sup_{t\in[0,T]}|M^{\psi_\epsilon}_t-\nu^{\psi_\epsilon}_t|\geq\xi\right) \nonumber\\
            	\leq&\, P\left(\bigg(\sup_{t\in[0,T]}|M^{\psi_\epsilon}_{t\wedge\tau_\epsilon^a}-\nu^{\psi_\epsilon}_{t\wedge\tau_\epsilon^a}|\geq\xi\bigg)\cap(\tau_\epsilon^a\geq T)\right)+P(\tau_\epsilon^a< T) \nonumber\\
            	\leq&\,\, \frac{1}{\xi^2}\mathbb{E}\left(\sup_{t\in[0,T]}|M^{\psi_\epsilon}_{t\wedge\tau_\epsilon^a}-\nu^{\psi_\epsilon}_{t\wedge\tau_\epsilon^a}|^2\right)+\frac{C}{a^2}
            \end{align}
            Let $\epsilon\to 0$ and $a\to\infty$, we get the desired result
            \begin{align*}
            	\lim\limits_{\epsilon\to 0}P\left(\sup_{t\in[0,T]}|M^{\psi_\epsilon}_t-\nu^{\psi_\epsilon}_t|\geq\xi\right)=0.
            \end{align*}
            
        \end{proof}

    \end{proposition}

\end{document}